\documentclass[a4paper]{amsart}

\theoremstyle{plain}
\usepackage{amsthm,amsfonts,amsmath,amscd,amssymb,amsxtra,amsbsy}
\usepackage{mathrsfs,mathtools}
\usepackage{color,colortbl}
\usepackage{float}
\usepackage{pdflscape}
\usepackage{slashed}

\usepackage{varioref}
\usepackage[bookmarksnumbered,colorlinks,plainpages]{hyperref}
\usepackage{cite}

\usepackage[all,2cell]{xy}
\usepackage{pstricks,latexsym,verbatim,graphicx}
\usepackage{multirow}
\usepackage{stmaryrd}
\usepackage{makecell}

\usepackage{tikz}
\usepackage{yhmath}

\usetikzlibrary{positioning,arrows}
\usetikzlibrary{decorations.pathreplacing}
\newtheorem*{Thm-A}{Theorem A}
\newtheorem*{Thm-B}{Theorem B}

\newtheorem{thm}{Theorem}[section]

\newtheorem{cor}[thm]{Corollary}
\newtheorem{lem}[thm]{Lemma}

\newtheorem{prop}[thm]{Proposition}

\theoremstyle{definition}
\newtheorem{defn}[thm]{Definition}
\newtheorem{rem}[thm]{\bf Remark}
\newtheorem{exm}[thm]{Example}

\numberwithin{equation}{section}

\def\Hom{{\rm Hom}}
\def\End{{\rm End}}
\def\Ext{{\rm Ext}}

\def\co{{\mathcal O}}

\def\coh{{\rm coh}\mbox{-}}
\def\vect{{\rm vect}\mbox{-}}
\def\bbY{{\mathbb Y}}
\def\bbX{{\mathbb X}}
\def\bbL{{\mathbb L}}

\def\bbZ{{\mathbb Z}}

\begin{document}
\title[Tilting Objects via Recollements and $p$-Cycles on Weighted Projective Lines]{ Tilting Objects via Recollements and $p$-Cycles on Weighted Projective Lines}

\author[Q. Dong, H.X. Zhang] {Qiang Dong, Hongxia Zhang$^*$}

\thanks{$^*$ the corresponding author}

\dedicatory{}
\commby{}

\begin{abstract}
In this paper, we provide a new method for constructing tilting objects in a triangulated category via recollements. The $p$-cycle approach to exceptional curve processes  significant advantages in constructing recollements and ladders, due to the existence of reduction/insertion functors. In order to construct tilting objects in the stable category of vector bundles over a weighted projective line, we give explicit expressions for line bundles and extension bundles due to the $p$-cycles constuctions. Furthermore, we provide an essential proof for tilting cuboic object and tilting objects consisting of Auslander bundles. Moreover, we construct certain new tilting objects in the stable category of vector bundles over a weighted projective line.
\end{abstract}

\maketitle
\section{Introduction}

Geigle and Lenzing introduced the notion of weighted projective lines in \cite{[GL1]}, which provide a geometric approach to canonical algebras in the sense of Ringel \cite{[Rin]}. The subcategory $\textup{vect}\mbox{-}\mathbb{X}$ of vector bundles over a weighted projective line $\mathbb{X}$ with a distinguished exact structure is a Frobenius category whose indecomposable projective-injective objects are just all line bundles. It is well known that the induced stable category $\underline{\textup{vect}}\mbox{-}\mathbb{X}$ is a triangulated category by \cite{[Hap1988]}. Recently, many tilting objects in this triangulated category have been found; see \cite{CLLR,CLR2019,[DR],[KLM1],[L1]}. Accordingly, the study of the tilting theory in the stable category $\underline{\textup{vect}}\mbox{-}\mathbb{X}$ has a high contact with many mathematical subjects. Among the many related subjects we mention Nakayama algebra \cite{[DLR],[LMR]}, singularity theory\cite{[EP], [L1]}, monomorphism category \cite{[KLM2],[RS],[Simson2015]}, mirror symmetry\cite{[ET]}
and recollement theory\cite{[LiuLu]}.

Recollements of abelian and triangulated categories were introduced by Be\u{\i}linson, Bernstein, and Deligne \cite{[BBD]} in connection with derived categories of sheaves on topological spaces with the idea
that one triangulated category may be ¡°glued together¡± from two others.
Recollements play important roles in various branches of mathematics. For example, recollements of triangulated categories are used to stratify derived categories of sheaves by Cline-Parshall-Scott in \cite{[CPS2]} and to study highest weight categories in \cite{[CPS1],Krause2017HWC}. Meanwhile, recollements of abelian categories have also been investigated  extensively in various contexts; see \cite{[CK],[GKP],[P]}.

The concept of recollements was generalized to the idea of ladders by Beilinson-Ginsburg-Schechtman in \cite{[BGS]}, which have been used to study the derived simplicity of the derived module categories (see \cite{[AKLY]}) and the compactly generated triangulated categories (see \cite{[BDS],[GP]}). J{\o}rgensen showed that a recollement of a triangulated category with Serre duality can be extended to a ladder by using the reflecting approach, see \cite{[J]}.
For the stable category of vector bundles $\underline{\textup{vect}}\mbox{-}\mathbb{X}$, Ruan \cite{[R]}  constructed ladders by means of reduction and insertion functors arising from the $p$-cycle construction on weighted projective lines, c.f. \cite{[C]}.

Note that a recollement can be regarded as a short exact sequence (or a semi-orthogonal decomposition) of either a triangulated category or an abelian category. It offers a powerful method for investigating the `middle' large category by studying its two associated smaller ones. Gluing techniques have been widely employed in the literature to construct
torsion pairs, tilting objects, silting objects, t-structures and co-t-structures, see \cite{[AKL], [BBD], [Bo], [La], [LV], [LVY], MH, [MZ]}. The main purpose of this paper is to construct tilting objects in the stable category of vector bundles $\underline{\textup{vect}}\mbox{-}\mathbb{X}$ via recollements and ladders, due to the explicit description on the associated functors via the $p$-cycle construction.

The following is our main
theorem concerning tilting (see Theorem \ref{tilting in recollement}).

\begin{Thm-A}[\textup{see Theorem \ref{tilting in recollement}}]
Let \begin{equation}
\xymatrix@C=1cm{
\mathcal{C}'\ar[rrr]|{i_*=i_!}
&&& \mathcal{C}
\ar@/_1.5pc/[lll]|{i^*}
\ar@/^1.5pc/[lll]|{i^!}
\ar[rrr]|{j^!=j^*}
&&& \mathcal{C}''
\ar@/_1.5pc/[lll]|{j_!}
\ar@/^1.5pc/[lll]|{j_*}
}\end{equation}
be a recollement of triangulated categories. Assume $j_{*}$ admits a right adjoint $j^\sharp$.
Assume $T'$ and $T''$ are tilting objects in $\mathcal{C}'$ and $\mathcal{C}''$, respectively. Then the following statements are equivalent:
\begin{enumerate}
\item[(1)] $T=i_*(T') \oplus j_*(T'')$ is a tilting object in $\mathcal{C}$.
\item[(2)] $\Hom_{\mathcal{C}''}(T'', j^\sharp i_*(T')[n])=0$ for all nonzero integers $n\neq 0$.

\item[(3)] $\Hom_{\mathcal{C}'}(i^*j_*(T''), T'[n])=0$ for all nonzero integers $n\neq 0$.
\end{enumerate}
\end{Thm-A}

In order to apply Theorem A to construct tilting objects in $\underline{\textup{vect}}\mbox{-}\mathbb{X}$, we provide explicit $p$-cycle interpretation for line bundles and extension bundles (see Propositions \ref{$p$-cycle interpretation for line bundles} and \ref{p cycle interpretation for extension bundle}), and describe the actions of the reduction and insertion functors on these bundles with closed formulas (see Lemma \ref{ActOnL}, Propositions \ref{element reduction functor} and \ref{extension bundles insertion functor}). Combining with Theorem A with
the ladder given in \cite[Theorem 4.9]{[R]}, we develop new proofs for certain known tilting cuboic object and tilting objects consisting of Auslander bundles (see Propositions \ref{CubTilting} and  \ref{DRTilting}). Moreover, we construct  new tilting objects in $\underline{\vect}\bbX(2, p_2, p_3)$, see the following main theorem.

\begin{Thm-B}[\textup{see Theorem \ref{newtiltingobject}}]
Let $q$ be an integer with $1\leq q\leq p_3-2$. For $k=1,2$, we denote by
$$T_{1k}=\bigoplus_{i=0}^{q-1}\bigoplus_{j=0}^{p_2-2}E(i\overline{x}_k+j\overline{x}_3-\vec{x}_3+\vec{c})$$
and
$$T_{2k}=\bigoplus_{j=0}^{p_2-2}[E\langle q\vec{x}_3\rangle\big(j\overline{x}_3-(q+1)\vec{x}_3+\vec{c}\big)\bigoplus\bigoplus_{i=1}^{p_3-q-2}E\big(i\overline{x}_k+j\overline{x}_3-(q+1)\vec{x}_3+\vec{c}\big)]$$
in $\underline{\vect}\bbX(2, p_2, p_3)$. Then
$T_{1i}\oplus T_{2j}, 1\leq i, j\leq 2$ are tilting objects in  $\underline{\vect}\bbX(2, p_2, p_3)$.
\end{Thm-B}

This paper is structured as follows.
In the next section we recall the definition of the $p$-cycle construction for an exceptional curve. In section 3,
we give the $p$-cycle interpretation for line bundles and extension bundles over a weighted projective line. In Section 4, we review two important functors---reduction functor and insertion functor, and describe how they act on line bundles and extension bundles. In Sections 5 and 6, we give an effective method to construct tilting objects in a triangulated category, see Theorem \ref{tilting in recollement}, and realize the tilting objects in $\underline{\vect}\bbX(2, p_2, p_3)$ by our method, which are constructed in \cite{[DR],[KLM1]}. As an application, we provide the certain new tilting objects in $\underline{\vect}\bbX(2, p_2, p_3)$, see Theorem \ref{newtiltingobject}.

\section{Preliminary}

In this section, we introduce the category of coherent sheaves on an exceptional curve by using $p$-cycle construction due to Lenzing \cite{[L]}; see also \cite{[K],[Kussin]}. Throughout this paper we fix an arbitrary field $\mathbf{k}$, and denote by $D:=\Hom_{\mathbf{k}}(-, \mathbf{k})$.

\subsection{Exceptional curve}
In this subsection we collect some facts about exceptional curves from \cite[Section 2.5]{[L]}. For further investigation on the related topics we refer to \cite{[Kussin],[K],[Ld],[LR],[KM]}.

An object $E$ in a small abelian $\mathbf{k}$-category is called \emph{exceptional} if the endomorphism algebra of
$E$ is a skew field and moreover $E$ has no self-extensions, i.e. $\Ext^n(E, E) = 0$ for all $n>0$.

An \emph{exceptional curve} $\bbX$ over $\mathbf{k}$ is defined by the following requests on its associated
category $\mathcal{H}=\coh\bbX$ of coherent sheaves:

\begin{itemize}
  \item[(1)] $\mathcal{H}$ is a connected small abelian $\mathbf{k}$-category with morphism spaces that are finite dimensional
  over $\mathbf{k}$;
  \item[(2)] $\mathcal{H}$ is hereditary and noetherian, and there exists an equivalence $\tau:\mathcal{H}\to\mathcal{H}$
  such that Serre duality $D\Ext_\mathcal{H}^1(X,Y)\cong\Hom_\mathcal{H}(Y,\tau X)$ holds;
  \item[(3)] $\mathcal{H}$ admits a complete exceptional sequence $(E_1, . . . , E_n)$, i.e., each $E_i$ is exceptional and $\Hom(E_j, E_i) = 0=\Ext^1(E_j, E_i)$ for all $j>i$, and $E_1, . . . , E_n$ generate the bounded derived category of $\mathcal{H}$ as a triangulated category.
\end{itemize}

In case $\mathbf{k}$ is algebraically closed,
by \cite{[L3]} the exceptional curves are just the weighted projective lines.

Each indecomposable object in $\mathcal{H}=\coh\bbX$ is either a vector bundle or a torsion sheaf. The torsion subcategory $\mathcal{H}_0$ of $\mathcal{H}$ has a decomposition $\mathcal{H}_0=\coprod_{x\in \bbX}\mathcal{U}_x$,
where $\mathcal{U}_x$ is consisting of torsion sheaves concentrated in $x$.
Let $\mathcal{S}_x$ be the semisimple subcategory of $\mathcal{U}_x$. Each point $x\in\bbX$ determines
(by means of a mutation with respect to $\mathcal{S}_x$) a self-equivalence functor $\sigma_x: \mathcal{H}\to \mathcal{H}; \; E\mapsto E(x)$,
together with a natural transformation $x: \textup{Id}\to \sigma_x$, also denoted by the symbol $x$.
For each coherent sheaf $E$ with $\Hom_{\mathcal{H}}(\mathcal{U}_x, E)=0$, these data are given by the $\mathcal{S}_x$-universal extension:
\begin{equation}\label{universal extension}\xymatrix{0\ar[r]& E\ar[r]^{x_E}& E(x)\ar[r]& E_x \ar[r]& 0,}
\end{equation}
where
$E_x=\bigoplus_{S\in\mathcal{S}_x} \Ext^1(S, E)\otimes_{\End(S)} S$ belongs to $\mathcal{S}_x$, cf. \cite[Section 2.5]{[K]}. For an indecomposable torsion sheaf $E$, $\sigma_x$ acts as follows:
if $E\in\mathcal{U}_y$ for some $y\neq x$, then we have $E(x)=E$ and $x_E={\rm id}_E$ by the $\mathcal{S}_x$-universal extension (\ref{universal extension}) as above; if $E\in\mathcal{U}_x$, then $E(x)=\tau^{-1}E$
and the kernel of $x_E: E\to E(x)$ equals the simple socle of $E$.

\subsection{\texorpdfstring{$p$}{$p$}-cycle construction}

Fix a point $x$ of $\mathbb{X}$ and an integer $p\geq1$. A \emph{$p$-cycle} $E$ concentrated in $x$ is a diagram
$$\xymatrix{
\cdots\ar[r] & E_{n}\ar^{x_{n}}[r] &  E_{n+1}\ar^{x_{n+1}}[r] &\cdots \ar[r] & E_{n+p}\ar^{x_{n+p}}[r]&\cdots},$$
which is $p$-periodic in the sense that $E_{n+p}=E_{n}(x)$, $x_{n+p}=x_{n}(x)$, and $x_{n+p-1}\cdots x_{n+1} x_{n}=x_{E_{n}}$ holds for each integer $n$. A morphism $u: E\to F$ of $p$-cycles concentrated in the same point $x$ is a sequence of morphism $u_n: E_n\to F_n$ which is $p$-periodic, i.e, $u_{n+p}=u_{n}(x)$ for each $n$ and such that each diagram
$$\xymatrix{
E_n\ar[r]^{x_n}\ar[d]^{u_n} &E_{n+1}\ar[d]^{u_{n+1}}\\
F_n\ar[r]^{x_n}& F_{n+1}
}$$
commutes.
We denote a $p$-cycle in the form
\begin{equation}\label{p-cycle}\xymatrix{
E_0\ar[r]^{x_0} & E_{1}\ar[r]^{x_{1}} & \cdots \ar[r]^{x_{p-2}}&  E_{p-1}\ar[r]^{x_{p-1}} & E_{0}(x),
}\end{equation}
and denote the category of all $p$-cycles concentrated in $x$ by $\mathcal{H}(p;x)$.

According to \cite[Section 4]{[L]}, the category $\mathcal{H}(p;x)$ is connected, abelian and noetherian, where exactness and formation of kernels and cokernels have a pointwise interpretation. Moreover, it is equivalent to the category of coherent sheaves on an exceptional curve which admits Serre duality.
Note that any functor $\sigma$ on $\mathcal{H}$ induces a functor on $\mathcal{H}(p;x)$ via pointwise actions, which will be still denoted by $\sigma$.
Moreover, $\mathcal{H}(p;x)$ is again equipped with a natural shift automorphism $\overline{\sigma}_{x}$ satisfying $\overline{\sigma}_{x}^p=\sigma_x$ (see \cite[Section 6.1.5]{[Kussin]}),
which sends a $p$-cycle of the form (\ref{p-cycle}) to the following:
\begin{equation}\label{natural shift automorphism}\xymatrix{
 E_{1}\ar[r]^{x_{1}} & E_2\ar[r]^{x_2} &\cdots \ar[r]^{x_{p-1}}&  E_{0}(x)
\ar^{x_{0}(x)}[r]&  E_{1}(x).}\end{equation}

Let $E$ be a $p$-cycle in $\mathcal{H}(p;x)$ of the form (\ref{p-cycle}).
If $E$ is a vector bundle, then each $x_j$ is a monomorphism since each $x_{E_j}: E_j\to E_j(x)$ is a monomorphism by (\ref{universal extension}).

There is a full exact embedding
\begin{equation}\label{embedding}
\iota: {\mathcal{H}}\to \mathcal{H}(p;x); \quad E\mapsto (E=\cdots =E\to E(x)).
\end{equation}
We will identify $\mathcal{H}$ with the resulting exact subcategory of $\mathcal{H}(p;x)$.
We note that the inclusion $\iota: {\mathcal{H}}\to \mathcal{H}(p;x)$ has a left adjoint $l$ and a right adjoint $r$ which are both exact
functors and are given by
$${l}\xymatrix{
\big(E_0\ar[r]^{x_0} & E_{1}\ar[r]^{x_{1}} & \cdots \ar[r]^{x_{p-2}}&  E_{p-1}\ar[r]^{x_{p-1}} & E_{0}(x)\big)
}=E_{p-1}$$
and
$$r \xymatrix{
\big(E_0\ar[r]^{x_0} & E_{1}\ar[r]^{x_{1}} & \cdots \ar[r]^{x_{p-2}}&  E_{p-1}\ar[r]^{x_{p-1}} & E_{0}(x)\big)
}=E_{0}.$$

\subsection{Simple sheaves in \texorpdfstring{$\mathcal{H}(p;x)$}{$\mathcal{H}(p;x)$}}\label{section of simples}

The simple objects in $\mathcal{H}(p;x)$ occur in two types:

\begin{itemize}
  \item [(1)] the simple objects of $\mathcal{H}$ which are concentrated in a point $y$ different from $x$;
  \item [(2)] for each simple object $S$ from $\mathcal{H}$ which is concentrated in $x$, the $p$ simples:
  \begin{itemize}
  \item[] \qquad $S_1:\; 0\to 0\to \cdots \to 0\to S\to 0;$
  \item[] \qquad $S_2:\; 0\to 0\to \cdots \to S \to 0\to 0;$
  \item[] \qquad $\qquad\qquad\qquad\cdots\cdots$
  \item[] \qquad $S_{p-1}:\; 0\to S\to \cdots \to 0 \to 0\to 0;$
  \item[] \qquad $S_{p}: \;S\to 0\to \cdots \to 0 \to 0\to S(x).$
\end{itemize}
\end{itemize}
Moreover, each $S_{j}\; (1\leq j\leq p)$ is exceptional and $\End_{\mathcal{H}(p;x)}(S_j)\cong\End_{\mathcal{H}}(S)$. Recall that $S(x)=\tau^{-1}S$
for any $S\in\mathcal{S}_x$. The non-zero extensions between the simples of type (2) are given by
$$\Ext_{\mathcal{H}(p;x)}^1(S_{j+1}, S_j)\cong  \Ext_{\mathcal{H}(p;x)}^1(S(x)_{1},S_p)\cong\End_{\mathcal{H}}(S),$$ where $1\leq j\leq p-1$.
Consequently, $\bar{\tau}^{-1}(S_{j})=S_{j+1}$ for $1\leq j\leq p-1$ and $\bar{\tau}^{-1}(S_{p})=S(x)_1$, where $\bar{\tau}$ is the Auslander--Reiten translation of $\mathcal{H}(p;x)$.
In particular, if there is a unique simple sheaf $S$ in $\mathcal{U}_x$, i.e, $S(x)=S$, then $\bar{\tau}(S_{j+1})=S_j$, where the indices are taken modulo $p$. In this case, the extension-closed subcategory generated by the simples of type (2) is a tube of rank $p$.

\section{\texorpdfstring{$p$}{$p$}-cycle construction for weighted projective lines}

We first recall the alternative description of weighted projective lines in the sense of Geigle-Lenzing \cite{[GL1]} by $p$-cycle construction due to Lenzing \cite{[L]}.

Let $\bbX_0$ be the usual projective line over $\mathbf{k}$, and let ${\boldsymbol\lambda}=(\lambda_1,\lambda_2,\cdots,\lambda_t)$ be a sequence of pairwise distinct points from $\bbX_0$ and ${\mathbf{p}}=(p_1, p_2,\cdots, p_t)$ be a sequence of positive integers. Associated to these data there is a weighted projective line $\bbX({\mathbf{p}}, {\boldsymbol\lambda})$ with the weight sequence ${\mathbf{p}}$ and parameter sequence ${\boldsymbol\lambda}$, we refer to \cite{[GL1]} for details. Denote by ${\rm coh}\mbox{-}\mathbb{X}(\mathbf{p}, {\boldsymbol\lambda})$ the category of coherent sheaves on the weighted projective line $\mathbb{X}(\mathbf{p}, {\boldsymbol\lambda})$.

Let inductively $\bbX_i$ be the exceptional curve obtained from $\bbX_{i-1}$ by inserting weight $p_i$ in $\lambda_{i}$, i.e, forming the category of $p_i$-cycles in $\coh\bbX_{i-1}$ which are concentrated in $\lambda_i$. Then $\bbX_t$ is isomorphic to the weighted projective line $\bbX(\mathbf{p}, {\boldsymbol\lambda})$ (this statement has been shown in \cite[Section 4]{[L]}, see also \cite[Section 6.2.2]{[Kussin]} and \cite{Hubery} for alternative proofs). In the following we identify $\bbX_t$ with $\bbX(\mathbf{p}, {\boldsymbol\lambda})$.

From this construction, we see that for any $1\leq i, i'\leq t$, if $i'\leq i-1$, then there is a unique simple sheaf in $\bbX_{i'}$ concentrated in $\lambda_i$, which will be denoted by $S_{\lambda_i}$; if $i'\geq i$, then there are $p_i$-simples in $\bbX_{i'}$ concentrated in $\lambda_i$, which will be denoted by $S_{\lambda_i,j}$'s, where the index $j\in\bbZ$ is taken modulo $p_i$.

\subsection{\texorpdfstring{$p$}{$p$}-cycle interpretation for line bundles}
Let $\mathbf{p}=(p_1, p_2, \cdots, p_t)$ be a sequence of integers.
The \emph{string group} $\bbL(\mathbf{p})$ of type $\mathbf{p}$ is an abelian group (written additively) on generators $\vec{x}_1, \vec{x}_2, \cdots, \vec{x}_t$, subject to the relations $p_1\vec{x}_1=p_2\vec{x}_2=\cdots = p_t\vec{x}_t$, where this common element is denoted by $\vec{c}$ and called the \emph{canonical element} of $\bbL(\mathbf{p})$. Each element $\vec{x}$ in $\bbL(\mathbf{p})$ can be uniquely written in \emph{normal form}
$$\vec{x}=\sum_{i=1}^t l_i\vec{x}_i+l\vec{c},
\notag$$
where $0\leq l_i\leq p_i-1$ for $1\leq i\leq t$ and $l\in \mathbb{Z}$. The \emph{dualizing element} $\vec{\omega}$ in $\bbL(\mathbf{p})$ is defined as $\vec{\omega}=(t-2)\vec{c}-\sum_{i=1}^t \vec{x}_i$.

Each line bundle $L$ in ${\rm coh}\mbox{-}\mathbb{X}(\mathbf{p}, {\boldsymbol\lambda})$ has the form $\co(\vec{x})$ for a uniquely defined $\vec{x}\in \bbL(\mathbf{p})$, see \cite{[L1]}. Now we recall some important short exact sequences in ${\rm coh}\mbox{-}\mathbb{X}(\mathbf{p}, {\boldsymbol\lambda})$. For an ordinary point
$\lambda\in \bbX_0\backslash \{\lambda_1,\lambda_2,\cdots,\lambda_t\}$, define $S_{\lambda}$ by the following exact sequence:
\begin{equation}\label{important short exact sequences1}\xymatrix{0\ar[r]& \co\ar[r]^{x_2^{p_2}-\lambda x_1^{p_1}}& \co({\vec{c}})\ar[r]& S_{\lambda}\ar[r]& 0}.
\end{equation}
For an exceptional point $\lambda_i$ and $j=1,2,\cdots,p_i,$ define $S_{i,j}$ by the following exact sequence:
\begin{equation}\label{important short exact sequences2}\xymatrix{0\ar[r]& \co((j-1)\vec{x}_i)\ar[r]^{\;\;\;x_i}& \co(j\vec{x}_i)\ar[r]& S_{i,j}\ar[r]& 0}.
\end{equation}

The following proposition gives the $p$-cycle interpretation for line bundles in ${\rm coh}\mbox{-}\mathbb{X}(\mathbf{p}, {\boldsymbol\lambda})$.

\begin{prop}\label{$p$-cycle interpretation for line bundles}
Assume that $\vec{y}$ has the normal form $\vec{y}=\sum_{i=1}^{t}k_i\vec{x}_i+k\vec{c}$.
Denote by $\vec{y}^{\prime}=\sum_{i=1}^{t-1}k_i\vec{x}_i+k\vec{c}$. Then $\co(\vec{y})\in {\rm coh}\mbox{-}\mathbb{X}(\mathbf{p}, {\boldsymbol\lambda})$ corresponds to the follwing $p_t$-cycle
$$\xymatrix@C=1.268cm{
\co(\vec{y}^{\prime})\ar@{=}[r]& \cdots\ar@{=}[r] &\co(\vec{y}^{\prime})\ar[r]^{y_{p_t-k_t-1}}& \co(\vec{y}^{\prime}+\vec{c})\ar@{=}[r]& \cdots\ar@{=}[r]&\co(\vec{y}^{\prime}+\vec{c}),}$$
in $\mathcal{H}(p_t; \lambda_t)$, where $\mathcal{H}={\rm coh}\mbox{-}\bbX(p_1, p_2, \cdots, p_{t-1};\lambda_1, \lambda_2, \cdots, \lambda_{t-1})$.
\end{prop}

\begin{proof}
Recall from (\ref{embedding}) that there exists a full exact embedding
$$\iota: {\mathcal{H}}\to \mathcal{H}(p_t;\lambda_t); \quad E\mapsto (E=\cdots =E\to E(x)).$$
We obtain $$\iota(\co(\vec{y}^{\prime}))=\big(\co(\vec{y}^{\prime})=\cdots =\co\big(\vec{y}^{\prime})\to \co(\vec{y}^{\prime}+\vec{c})\big).$$
Consider the following commutative diagram (each column is exact sequence)
\small{$$\xymatrix{
\co(\vec{y}^{\prime})\ar@{=}[r]\ar[d]^{u_0}& \cdots\ar@{=}[r]&\co(\vec{y}^{\prime})\ar@{=}[r]\ar[d]^{u_{p_t-j-2}}&\co(\vec{y}^{\prime})\ar[r]\ar[d]^{u_{p_t-j-1}}&\co(\vec{y}^{\prime}+\vec{c})\ar@{=}[r]\ar[d]^{u_{p_t-j}}& \cdots\ar@{=}[r]&\co(\vec{y}^{\prime}+\vec{c})\ar[d]^{u_{p_t}}
\\
\co(\vec{y}^{\prime})\ar@{=}[r]\ar[d]^{0}& \cdots\ar@{=}[r]&\co(\vec{y}^{\prime})\ar[r]\ar[d]^{0}&\co(\vec{y}^{\prime}+\vec{c})\ar@{=}[r]\ar[d]^{v_{p_t-j-1}}& \co(\vec{y}^{\prime}+\vec{c})\ar@{=}[r]\ar[d]^{0}& \cdots\ar@{=}[r]&\co(\vec{y}^{\prime}+\vec{c})\ar[d]^{0}
\\
0\ar@{=}[r]& \cdots\ar@{=}[r]&0\ar[r]&S_{\lambda_t}\ar[r]&0\ar@{=}[r]& \cdots\ar@{=}[r]&0.
}$$}
Here, the third row is the simple $S_{\lambda_t,j+1}$. If $j=0$, the first row is just $\iota(\co(\vec{y}^{\prime}))$. Note that there exists only one simple sheaf $S_{t,1}$ for $\lambda_t$ such that the extension of $S_{t,1}$ by $\co(\vec{y}^{\prime})$ is nonsplit. In this case, $\textup{dim}_{{\bf k}}\textup{Ext}^1(\co(\vec{y}
^{\prime})),S_{t,1})=1$. It implies that $S_{\lambda_t,1}=S_{t,1}$. Combining with (\ref{important short exact sequences2}), we obtain that the second row in the diagram for $j=0$ corresponds to $\co(\vec{y}^{\prime}+\vec{x}_3)$, which is also the first row in the diagram for $j=1$. By induction on $j$, we get $S_{\lambda_t,j}=S_{t,j}$ and the result holds for $j=1,2,\cdots,p$.
\end{proof}

\subsection{\texorpdfstring{$p$}{$p$}-cycle interpretation for extension bundles}
In this subsection, we state the $p$-cycle interpretation for extension bundles in ${\rm coh}\mbox{-}\mathbb{X}(\mathbf{p}, {\boldsymbol\lambda})$ with three weights. We simply write $\bbX(\mathbf{p},{\boldsymbol\lambda})=\bbX(\mathbf{p})$ when $t\leq 3$, under the normalization of ${\boldsymbol\lambda}$ in the sense that $\lambda_1=\infty, \lambda_2=0,\lambda_3=1.$

\begin{defn}\textup{\cite[Definition 4.1]{[KLM1]}}
For a line bundle $L\in {\rm coh}\mbox{-}\bbX(p_1, p_2, p_3)$ with $p_i\geq 2, i=1,2,3$, and $\vec{x}=\sum_{i=1}^{3}l_i\vec{x}_i$ with $0\leq l_i\leq p_i-2$, let
\begin{equation}\label{extension bundle}\xymatrix{0\ar[r]& L(\vec{\omega})\ar[r]^{f}& E\ar[r]^{g}& L(\vec{x}) \ar[r]& 0,}
\end{equation}
be nonsplit exact sequence. The central term $E_{L}\langle \vec{x}\rangle:=E$ of the sequence, which is uniquely defined up to isomorphism, is called the \emph{extension bundle} given by the data $(L,\vec{x})$. If $L=\co,$ then we just write $E\langle \vec{x}\rangle$. For $\vec{x}=0,$ the sequence
$$\xymatrix{0\ar[r]& L(\vec{\omega})\ar[r]^{f}& E_{L}\langle 0\rangle\ar[r]^{g}& L \ar[r]& 0}$$ is almost split, and $E(L):=E_{L}\langle 0\rangle$ is called the \emph{Auslander bundle} associated with $L$.
\end{defn}

For simplicity, we define a mapping $\phi: \bbL(p_1, p_2, p_3)\rightarrow \bbL(p_1, p_2)$ by sending $\vec{x}$ with the normal form $\vec{x}=\sum_{i=1}^{3}l_i\vec{x}_i+l\vec{c}$ to $\phi(\vec{x})=\sum_{i=1}^{2}l_i\vec{x}+l\vec{c}.$

The following result  presents the $p$-cycle interpretation for extension bundles in the category ${\rm coh}\mbox{-}\mathbb{X}(p_1, p_2, p_3)$.

\begin{prop}\label{p cycle interpretation for extension bundle}
Let $p_1,p_2,p_3$ be positive integers greater than $1$. Assume that $\vec{x}=\sum_{i=1}^3l_i\vec{x}_i$ with $0\leq l_i\leq p_i-2$ for $i=1,2,3$, and $\vec{y}=\sum_{i=1}^3k_i\vec{x}_i+k\vec{c}$ is in normal form. Then the extension bundle $E\langle\vec{x}\rangle(\vec{y})$ corresponds to the image of the following $p_3$-cycle under the shift action $\overline{\sigma}_x^{k_3}$ in
 $\mathcal{H}(p_3;x)$ with $\mathcal{H}={\rm coh}\mbox{-}\bbX(p_1, p_2)$,
$$\xymatrix@C=1.268cm{
E_0\ar[r]^{y_0}&E_1\ar@{=}[r] &\cdots\ar@{=}[r]& E_{p_3-l_3-1}\ar[r] ^{y_{p_3-l_3-1}}&E_{p_3-l_3}\ar@{=}[r]&\cdots\ar@{=}[r]\cdots\ar@{=}[r]& E_{0}(\vec{c}).}$$
Here
$$E_i=\left\{
\begin{array}{ll}
    \mathcal{O}(-\vec{x}_1+l_2\vec{x}_2+\phi(\vec{y}))\oplus\mathcal{O}(l_1\vec{x}_1-\vec{x}_2+\phi(\vec{y})) & i=0\\
\mathcal{O}\big(\phi(\vec{\omega})+\vec{c}+\phi(\vec{y})\big)\oplus\mathcal{O}\big(\phi(\vec{x})+\phi(\vec{y})\big) & 1\leq i\leq p_3-l_3-1, \\
\end{array}\right.$$
and
$$y_{0}=\begin{bmatrix}
x_2^{p_2-l_2-1}&-x_1^{p_1-l_1-1}\\
x_1^{l_1+1}&-x_2^{l_2+1}
\end{bmatrix}, \quad y_{p_3-l_3-1}=\begin{bmatrix}
x_2^{l_2+1}&-x_1^{p_1-l_1-1}\\
x_1^{l_1+1}&-x_2^{p_2-l_2-1}
\end{bmatrix}.$$
\end{prop}

\begin{proof}
First, we prove the result for the case $\vec{y}=0$. By Proposition \ref{$p$-cycle interpretation for line bundles}, we know that $\mathcal{O}(\vec{\omega})$ (resp. $\mathcal{O}(\vec{x})$) corresponds to the $p_3$-cycle which is the first row (resp. the third row) in the following diagram.

\begin{figure}[ht]
    \centering
\begin{tikzpicture}
\node (1) at (0,0) {\tiny{$\mathcal{O}(\phi(\vec{\omega}))$}};
\node (2) at (2,0) {\tiny{$\mathcal{O}(\phi(\vec{\omega})+\vec{c})$}};
\node (3) at (4,0) {\tiny{$\cdots$}};
\node (4) at (6,0) {\tiny{$\mathcal{O}(\phi(\vec{\omega})+\vec{c})$}};
\node (5) at (8.5,0) {\tiny{$\mathcal{O}(\phi(\vec{\omega})+\vec{c})$}};
\node (6) at (10.5,0) {\tiny{$\cdots$}};
\node (7) at (12.5,0) {\tiny{$\mathcal{O}(\phi(\vec{\omega})+\vec{c})$}};
\draw[->] (1) to (2);
\draw[-] (2.9,0.04) to (3.6,0.04);
\draw[-] (2.9,-0.04) to (3.6,-0.04);
\draw[-] (5.1,0.04) to (4.4,0.04);
\draw[-] (5.1,-0.04) to (4.4,-0.04);
\draw[-] (6.9,0.04) to (7.6,0.04);
\draw[-] (6.9,-0.04) to (7.6,-0.04);
\draw[-] (9.4,0.04) to (10.1,0.04);
\draw[-] (9.4,-0.04) to (10.1,-0.04);
\draw[-] (11.6,0.04) to (10.9,0.04);
\draw[-] (11.6,-0.04) to (10.9,-0.04);
\node (11) at (0,-1.2) {\tiny{$E_0$}};
\node (12) at (2,-1.2) {\tiny{$E_1$}};
\node (13) at (4,-1.2) {\tiny{$\cdots$}};
\node (14) at (6,-1.2) {\tiny{$E_{p_3-l_3-1}$}};
\node (15) at (8.5,-1.2) {\tiny{$E_{p_3-l_3}$}};
\node (16) at (10.5,-1.2) {\tiny{$\cdots$}};
\node (17) at (12.5,-1.2) {\tiny{$E_0(x)$}};
\node (21) at (0,-2.4) {\tiny{$\mathcal{O}(\phi(\vec{x}))$}};
\node (22) at (2,-2.4) {\tiny{$\mathcal{O}(\phi(\vec{x}))$}};
\node (23) at (4,-2.4) {\tiny{$\cdots$}};
\node (24) at (6,-2.4) {\tiny{$\mathcal{O}(\phi(\vec{x}))$}};
\node (25) at (8.5,-2.4) {\tiny{$\mathcal{O}(\phi(\vec{x})+\vec{c})$}};
\node (26) at (10.5,-2.4) {\tiny{$\cdots$}};
\node (27) at (12.5,-2.4) {\tiny{$\mathcal{O}(\phi(\vec{x})+\vec{c})$}};
\draw[-] (0.6,-2.36) to (1.4,-2.36);
\draw[-] (0.6,-2.44) to (1.4,-2.44);
\draw[-] (2.9,-2.36) to (3.6,-2.36);
\draw[-] (2.9,-2.44) to (3.6,-2.44);
\draw[-] (5.1,-2.36) to (4.4,-2.36);
\draw[-] (5.1,-2.44) to (4.4,-2.44);
\draw[->] (24) to (25);
\draw[-] (9.4,-2.36) to (10.1,-2.36);
\draw[-] (9.4,-2.44) to (10.1,-2.44);
\draw[-] (11.6,-2.36) to (10.9,-2.36);
\draw[-] (11.6,-2.44) to (10.9,-2.44);
\draw[->] (11) to (12);
\draw[->] (12) to (13);
\draw[->] (13) to (14);
\draw[->] (14) to (15);
\draw[->] (15) to (16);
\draw[->] (16) to (17);
\draw[->] (1) to (11);
\draw[->] (2) to (12);
\draw[->] (4) to (14);
\draw[->] (5) to (15);
\draw[->] (7) to (17);
\draw[->] (11) to (21);
\draw[->] (12) to (22);
\draw[->] (14) to (24);
\draw[->] (15) to (25);
\draw[->] (17) to (27);
\node () at (1,0.35) {\tiny{$x_2^{p_2}-x_1^{p_1}$}};
\node () at (7.2,-2.2) {\tiny{$x_2^{p_2}-x_1^{p_1}$}};
\node () at (0.25,-0.6) {\tiny{$u_0$}};
\node () at (2.25,-0.6) {\tiny{$u_1$}};
\node () at (6.7,-0.6) {\tiny{$u_{p_3-l_3-1}$}};
\node () at (9,-0.6) {\tiny{$u_{p_3-l_3}$}};
\node () at (12.8,-0.6) {\tiny{$u_{p_3}$}};
\node () at (0.25,-1.8) {\tiny{$v_0$}};
\node () at (2.25,-1.8) {\tiny{$v_1$}};
\node () at (6.7,-1.8) {\tiny{$v_{p_3-l_3-1}$}};
\node () at (9,-1.8) {\tiny{$v_{p_3-l_3}$}};
\node () at (12.8,-1.8) {\tiny{$v_{p_3}$}};
\node () at (1,-1) {\tiny{$y_{0}$}};
\node () at (3,-1) {\tiny{$y_{1}$}};
\node () at (7.3,-1) {\tiny{$y_{p_3-l_3-1}$}};
\node () at (11.5,-1) {\tiny{$y_{p_3-1}$}};
\end{tikzpicture}
\end{figure}
We assume that the diagram commutes and  each column is exact sequence. By Serre duality, we have
$$\Ext_{\mathcal{H}}^1(\co(\vec{z}_1), \co(\vec{z}_2))\cong D\Hom_{\mathcal{H}}(\co(\vec{z}_2),\co(\vec{z}_1-\vec{x}_1-\vec{x}_2)).$$
Hence, there exists ``unique'' non-split exact sequence
$$0\rightarrow\mathcal{O}(\phi(\vec{\omega}))\xrightarrow{\begin{bmatrix}
    X_2^{l_2+1}\\X_1^{l_1+1}
\end{bmatrix}}\mathcal{O}(-\vec{x}_1+l_2\vec{x}_2)\oplus\mathcal{O}(l_1\vec{x}_1-\vec{x}_2)\xrightarrow{\begin{bmatrix}
    X_1^{l_1+1}&-X_2^{l_2+1}
\end{bmatrix}}\mathcal{O}(\phi(\vec{x}))\rightarrow0$$
and no nonsplit extension of $\mathcal{O}(\phi(\vec{x}))$ by $\mathcal{O}(\phi(\vec{\omega})+\vec{c})$. It follows that
$$E_i=\left\{
\begin{array}{ll}
    \mathcal{O}(-\vec{x}_1+l_2\vec{x}_2)\oplus\mathcal{O}(l_1\vec{x}_1-\vec{x}_2) & i=0\\
\mathcal{O}\big(\phi (\vec{\omega})+\vec{c}\big)\oplus\mathcal{O}\big(\phi(\vec{x})\big) & 1\leq i\leq p_3-l_3-1 \\
    \mathcal{O}(-\vec{x}_1+l_2\vec{x}_2+\vec{c})\oplus\mathcal{O}(l_1\vec{x}_1-\vec{x}_2+\vec{c}) & p_3-l_3\leq i\leq p_3-1 \\
\end{array}\right.$$
Using the commutative property of the diagram, we obtain $y_i$ is identity, except
$$y_{0}=\begin{bmatrix}
x_2^{p_2-l_2-1}&-x_1^{p_1-l_1-1}\\
x_1^{l_1+1}&-x_2^{l_2+1}
\end{bmatrix}, \quad y_{p_3-l_3-1}=\begin{bmatrix}
x_2^{l_2+1}&-x_1^{p_1-l_1-1}\\
x_1^{l_1+1}&-x_2^{p_2-l_2-1}
\end{bmatrix}.$$

Note that $E\langle\vec{x}\rangle$ is the ``unique'' middle term of the nonsplit extension of $\mathcal{O}(\vec{x})$ by $\mathcal{O}(\vec{\omega})$. It implies that $E\langle\vec{x}\rangle$ corresponds to the $p_3$-cycle which is just the second row in the diagram.

Finally, $E\langle\vec{x}\rangle(\vec{y})$ is the image of $E\langle\vec{x}\rangle$ under the twist action $(\vec{y})$. Here, $\vec{y}=\phi(\vec{y})+k_3\vec{x}_3$. The twist action $\big(\phi(\vec{y})\big)$ on $\mathcal{H}(p_3; x)$ acts on a $p_3$-cycle via pointwise actions, and the action $(k_3\vec{x}_3)$ is just the shift action $\overline{\sigma}_x^{k_3}$. We finish the proof.
\end{proof}

\begin{rem}
Combining Proposition \ref{p cycle interpretation for extension bundle} with \cite[Proposition 2.3]{[DR]}, the Auslander bundle corresponds to the $p_3$-cycle
\begin{equation*}\xymatrix{
E_0\ar[r]^{y_0} & E_{1}\ar[r]^{y_{1}} & \cdots \ar[r]^{y_{p_3-2}}&  E_{p_3-1}\ar[r]^{y_{p_3-1}} & E_{0}(x),
}\end{equation*}
with all $y_i$ are identities except $y_j$ and $y_{j+1}$ for some $0\leq j\leq p_3-2$.
\end{rem}

\section{Reduction functor \texorpdfstring{$\psi^j$}{$\psi^j$} and insertion functor \texorpdfstring{$\psi_j$}{$\psi_j$}}
Throughout this section,
let $\bbX$ be an exceptional curve, and $\mathcal{H}$ the category of coherent sheave over $\bbX$.
We fix a point $x$ in $\bbX$. For any integer $p\geq 1$, denote by $\mathcal{H}(p):=\mathcal{H}(p;x)$ the category of all the $p$-cycles in $\mathcal{H}$ which are concentrated in $x$.

For any $0\leq j\leq p-1$, we define a pair of functors $(\psi^j, \psi_j)$ as follows.
The \emph{reduction functor} $\psi^j: \mathcal{H}(p)\to \mathcal{H}(p-1)$ is defined by sending a $p$-cycle $E$ of the form (\ref{p-cycle})
to a $(p-1)$-cycle:
$$\xymatrix{
E_0\ar[r]^{x_0} & \cdots\ar[r] & E_{j-1}\ar[r]^{x_jx_{j-1}} &  E_{j+1}\ar[r]^{x_{j+1}} &\cdots\ar[r] & E_{p-1}\ar[r]^{x_{p-1}} & E_{0}(x),
}
$$
and sending a morphism of $p$-cycles $(u_0,u_1,\cdots, u_{p-1})$ to
$(u_0,\cdots, u_{j-1}, u_{j+1}, \cdots, u_{p-1}).$
Conversely, the \emph{insertion functor} $\psi_j: \mathcal{H}(p-1)\to \mathcal{H}(p)$ is defined by sending a $(p-1)$-cycle $E'$ of the form
\begin{equation}\label{p-1 cycle}
\xymatrix{
E'_0\ar[r]^{x'_0} & E'_{1}\ar[r]^{x'_{1}} &\cdots\ar[r] & E'_{p-2}\ar[r]^{x'_{p-2}} & E'_{0}(x)
}\end{equation}
to a $p$-cycle
\begin{equation}\label{image of p-1 cycle}\xymatrix@C=0.6cm{
E'_0\ar[r]^{x'_0} & \cdots\ar[r] & E'_{j}\ar@{=}[r]  &E'_{j}\ar[r]^{x'_j}  & E'_{j+1}\ar[r]^{x'_{j+1}} &\cdots\ar[r] & E'_{p-2}\ar[r]^{x'_{p-2}} & E'_{0}(x),}
\end{equation}
and sending a morphism of $(p-1)$-cycles $(u'_0,\cdots, u'_{j}, u'_{j+1}, \cdots, u'_{p-2})$ to a morphism of $p$-cycles $(u'_0,\cdots, u'_{j}, u'_{j}, u'_{j+1}, \cdots, u'_{p-2})$.

Roughly speaking, $\psi^j(E)$ is obtained from $E$ by deleting the item $E_j$ and taking the composition of the morphisms $\xymatrix{
E_{j-1}\ar[r]^{x_{j-1}}&E_{j}\ar[r]^{x_{j}} & E_{j+1}}$; conversely, $\psi_j(E')$ is obtained from $E'$ by inserting in the $j$-th position a copy of $E'_j$: $\xymatrix{
E'_{j}\ar@{=}[r]  &E'_{j}\ar[r]^{x'_j}  & E'_{j+1}}$.

Recall that the functor $\sigma_x: \mathcal{H}\to \mathcal{H}; \;E\mapsto E(x)$ induces a functor on $\mathcal{H}(p)$ (resp. $\mathcal{H}(p-1)$) via pointwise actions, which shares the same notation $\sigma_x$. Denote by $\overline{\sigma}_{x}$ and $\overline{\sigma}_{x}'$ the natural shift automorphisms on $\mathcal{H}(p)$ and $\mathcal{H}(p-1)$ respectively satisfying $$\overline{\sigma}_{x}^p=\sigma_x: \mathcal{H}(p)\to \mathcal{H}(p)\quad{\text {and}}\quad(\overline{\sigma}_{x}')^{p-1}=\sigma_x: \mathcal{H}(p-1)\to\mathcal{H}(p-1).$$
By the above constructions, we have $$\psi^{p-1}=\psi^0\overline{\sigma}_{x}^{-1}
\quad{\text {and}}\quad \psi_{p-1}=\overline{\sigma}_{x}\psi_0.$$
For any $n\in\bbZ$, define
\begin{equation}\label{def of fun for gen case} \psi^{np+j}=(\overline{\sigma}_{x}')^{-n}\psi^{j}
\text{\quad and \quad}\psi_{np+j}=\psi_{j}(\overline{\sigma}_{x}')^{n}.
\end{equation}
For any integer $j$, the functor $\psi^{j}$ is called a \emph{reduction functor} and
the functor $\psi_{j}$ is called an \emph{insertion functor}.

\subsection{The functors \texorpdfstring{$\psi^j,\psi_j$}{$\psi^j,\psi_j$} act on line bundles}
Let $$\psi^j: {\rm coh}\mbox{-}\bbX(p_1, p_2, p_3)\to {\rm coh}\mbox{-}\bbX(p_1, p_2, p_3-1),$$
$$\psi_j: {\rm coh}\mbox{-}\bbX(p_1, p_2, p_3)\to {\rm coh}\mbox{-}\bbX(p_1, p_2, p_3+1)$$ be the reduction functor and insertion functor. According to Proposition \ref{$p$-cycle interpretation for line bundles}, we get the following description directly.

\begin{lem}\label{ActOnL}
Let $\vec{y}=\sum_{i=1}^{3}k_i\vec{x}_{i}+k\vec{c}\in \bbL(p_1, p_2, p_3)$ be the normal form. Then
\begin{equation}
\psi^j(\co(\vec{y}))=
  \left\{
  \begin{array}{ll}
   \co(\vec{y})& \;\;0\leq j< p_3-k_3\nonumber\\
  \co(\vec{y}-\vec{x}_3)& \;\;p_3-k_3\leq j\leq p_3-1
  \end{array}
\right.\end{equation}
and
\begin{equation}
\psi_j(\co(\vec{y}))=
  \left\{
  \begin{array}{ll}
   \co(\vec{y})& \;\;0\leq j< p_3-k_3\nonumber\\
  \co(\vec{y}+\vec{x}_3)& \;\;p_3-k_3\leq j\leq p_3-1.
  \end{array}
\right.\end{equation}
\end{lem}

\begin{proof}
By Proposition \ref{$p$-cycle interpretation for line bundles}, $\co(\vec{y})$ corresponds to the following $p_3$-cycle
$$\xymatrix@C=1.268cm{
\co(\phi(\vec{y}))\ar@{=}[r]& \cdots\ar@{=}[r] &\co(\phi(\vec{y}))\ar[r]^{y_{p_3-k_3-1}}& \co(\phi(\vec{y})+\vec{c})\ar@{=}[r]& \cdots\ar@{=}[r]&\co(\phi(\vec{y})+\vec{c}).}$$
According to the definition of the reduction functor $\psi^j$, we know $\psi^j(\co(\vec{y}))=\co(\vec{y})$ in $\bbL(p_1, p_2, p_3-1)$ for $0< j< p_3-k_3$, and $\psi^j(\co(\vec{y}))=\co(\vec{y}-\vec{x}_3)$ for $p_3-k_3\leq j\leq p_3-1$. Moreover, $$\psi^0(\co(\vec{y}))=\overline{\sigma}'\psi^1\overline{\sigma}^{-1}(\co(\vec{y}))=\overline{\sigma}'\psi^1(\co(\vec{y}-\vec{x}_3)).$$
For $k_3=0$, we have
$$\psi^0(\co(\vec{y}))=\overline{\sigma}'\psi^1(\co(\vec{y}-\vec{c}+(p_3-1)\vec{x}_3))=\co(\vec{y}-\vec{c}+(p_3-1)\vec{x}_3))=\co(\vec{y}).$$
For $k_3\neq0$, it follows that
$$\psi^0(\co(\vec{y}))=\overline{\sigma}'\psi^1(\co(\sum_{i=1}^{2}k_i\vec{x}_{i}+k\vec{c}+(k_3-1)\vec{x}_3))=\co(\vec{y}).$$
Therefore,
\begin{equation}
\psi^j(\co(\vec{y}))=
  \left\{
  \begin{array}{ll}
   \co(\vec{y})& \;\;0\leq j< p_3-k_3\nonumber\\
  \co(\vec{y}-\vec{x}_3)& \;\;p_3-k_3\leq j\leq p_3-1
  \end{array}
\right.\end{equation}
The condition (2) can be proven by a similar calculation.
\end{proof}

\subsection{The reduction functor acts on extension bundles}

The reduction functor $\psi^j$ does not always preserve the extension bundles while it preserves the line bundles by the above calculation. The following proposition gives the image of the extension bundle under the reduction functor.

\begin{prop}\label{element reduction functor}
Assume that $\vec{x}=\sum_{i=1}^3l_i\vec{x}_i$ with $0\leq l_i\leq p_i-2$ for $i=1,2,3$, and $\vec{y}=\sum_{i=1}^3k_i\vec{x}_i+k\vec{c}$ is the normal form in $\bbL(p_1,p_2,p_3)$. Then
\begin{itemize}
   \item[(1)] For $(l_3,j)\neq(0,p_3-k_3),(p_3-2,1-k_3)$, then
    \begin{equation}
\psi^j(E\langle\vec{x}\rangle(\vec{y}))=
  \left\{
  \begin{array}{ll}
   E\langle\vec{x}\rangle(\vec{y})& \;\;-k_3< j< p_3-l_3-k_3\nonumber\\
  E\langle\vec{x}-\vec{x}_3\rangle(\vec{y})& \;\;p_3-l_3-k_3\leq j\leq p_3-k_3.
 \end{array}
\right.\end{equation}
\item[(2)] For $(l_3, j)=(0,p
_3-k_3)$, then $$\psi^{p_3-k_3}(E\langle\vec{x}\rangle(\vec{y}))=\co\big(-\vec{x}_1-\vec{x}_2+\vec{c}+\vec{y}-\vec{x}_3\big)\oplus\co\big(\vec{x}+\vec{y}-\vec{x}_3\big).$$

\item[(3)] For $(l_3, j)=(p_3-2,1-k_3)$, then$$\psi^{1-k_3}(E\langle\vec{x}\rangle(\vec{y}))=\co(\vec{c}-\vec{x}_1+l_2\vec{x}_2-\vec{x}_3+\vec{y})\oplus\co(\vec{c}+l_1\vec{x}_1-\vec{x}_2-\vec{x}_3+\vec{y}).$$
\end{itemize}
\end{prop}

\begin{proof}
First, we prove the result for the case $\vec{y}=0$. According to Proposition \ref{p cycle interpretation for extension bundle}, the extension bundle $E\langle\vec{x}\rangle$ corresponds to the following $p_3$-cycle
$$E_0\xrightarrow{y_0}E_1\xrightarrow{y_1}\cdots\rightarrow E_{p_3-1}\xrightarrow{y_{p_3-1}}E_0(x),$$
where $y_i=\textup{id}$ for $i\neq 0,p_3-l_3-1$ and
$$y_{0}=\begin{bmatrix}
x_2^{p_2-1}&-x_1^{p_1-1}\\
x_1&-x_2
\end{bmatrix},\quad y_{p_3-l_3-1}=\begin{bmatrix}
x_2&-x_1^{p_1-1}\\
x_1&-x_2^{p_2-1}
\end{bmatrix}.$$

(Case I) If $l_3\neq 0\textup{\;or\;}p_3-2$, then $\psi^{j}(E\langle\vec{x}\rangle)$ is obtained from the $p_3$-cycle of $E\langle\vec{x}\rangle$ by deleting an arrow which is identity. For $0<j\leq p_3-1_3-1$, $\psi^{j}(E\langle\vec{x}\rangle)=E\langle\vec{x}\rangle$. For $p_3-l_3-1<j\leq p_3$, $\psi^{j}(E\langle\vec{x}\rangle)=E\langle\vec{x}-\vec{x}_3\rangle$.

(Case II) If $l_3=0$ (resp. $p_3-2$), we have $\psi^{0}(E\langle\vec{x}\rangle)$ (resp. $\psi^{1}(E\langle\vec{x}\rangle)$) is the $p_3$-cycle
$$\begin{array}{c}
    E_1\xrightarrow{y_1}E_2\xrightarrow{y_2}\cdots \xrightarrow{y_{p_3-2}}E_{p_3-1}\xrightarrow{y_0y_{p_3-1}}E_{1}(x)  \\
   \big(\textup{resp.}\; E_0\xrightarrow{y_1y_0}E_2\xrightarrow{y_2}\cdots \xrightarrow{y_{p_3-2}}E_{p_3-1}\xrightarrow{y_{p_3-1}}E_{0}(x)\big),
\end{array}$$
which corresponds to $\co(\phi(\vec{\omega})+\vec{c})\oplus\co(\phi(\vec{x}))$ (resp. $\co(\vec{c}-\vec{x}_1+l_2\vec{x}_2-\vec{x}_3)\oplus\co(\vec{c}+l_1\vec{x}_1-\vec{x}_2-\vec{x}_3)$) in $\coh\bbX(p_1,p_2,p_3-1)$.

For $l_3=0$ and $1\leq j\leq p_3-1$, we have $\psi^{j}(E\langle\vec{x}\rangle)=E\langle\vec{x}\rangle$. For $l_3=p_3-2$ and $2\leq j\leq p_3$, we have $\psi^{j}(E\langle\vec{x}\rangle)=E\langle\vec{x}-\vec{x}_3\rangle$.

Finally, for general $\vec{y}$, we note that $\psi^j=\overline{\sigma}'\psi^{j+1}\overline{\sigma}^{-1}$. Hence, $$\psi^j(E\langle\vec{x}\rangle(k_3\vec{x}_3))=\psi^j\overline{\sigma}^{k_3}(E\langle\vec{x}\rangle)=(\overline{\sigma}')^{k_3}\psi^{j+k_3}(E\langle\vec{x}\rangle).$$
By a direct calculation, the result holds.
\end{proof}
For the case of the Auslander bundles,
the following corollary follows immediately by setting taking $\vec{x}=0$.

\begin{cor}\label{extension bundles 2 reduction functor}
Let $\vec{y}=\sum_{i=1}^3k_i\vec{x}_i+k\vec{c} \in\bbL(p_1,p_2,p_3)$ be its normal form. Then
\begin{equation}
\psi^j(E(\vec{y}))=
  \left\{
  \begin{array}{ll}
   E(\vec{y})& \;\;0< j< p_3-k_3\nonumber\\
  \co(\vec{y}+\vec{c}-\vec{x}_1-\vec{x}_3)\oplus\co(\vec{y}+\vec{c}-\vec{x}_2-\vec{x}_3)& \;\;j=p_3-k_3\nonumber\\
  E(\vec{y}-\vec{x}_3)& \;\;p_3-k_3<j\leq p_3.
  \end{array}
\right.\end{equation}
\end{cor}

\subsection{The insertion functor acts on extension bundles}
In this subsection, we exhibit that how the insertion functor acts on the extension bundles.

\begin{prop}\label{extension bundles insertion functor}
Assume that $\vec{x}=\sum_{i=1}^3l_i\vec{x}_i$ with $0\leq l_i\leq p_i-2$ for $i=1,2,3$ and $\vec{y}=\sum_{i=1}^3k_i\vec{x}_i+k\vec{c}$ is the normal form in $\bbL(p_1,p_2,p_3)$. Then
\begin{equation}
\psi_j(E\langle\vec{x}\rangle(\vec{y}))=
  \left\{
  \begin{array}{ll}
   E\langle\vec{x}\rangle(\vec{y})& \;\;-k_3< j< p_3-l_3-k_3\nonumber\\
  E\langle\vec{x}+\vec{x}_3\rangle(\vec{y})& \;\;p_3-l_3-k_3\leq j\leq p_3-k_3.
  \end{array}
\right.\end{equation}
\end{prop}

\begin{proof}
Let $E_0\xrightarrow{y_0}E_1\xrightarrow{y_1}\cdots\rightarrow E_{p_3-1}\xrightarrow{y_{p_3-1}}E_0(x)$ be the $p_3$-cycle of the extension bundle $E\langle\vec{x}\rangle$ similar as in Proposition \ref{element reduction functor}. By the definition of the insertion functor, $\psi_j(E\langle\vec{x}\rangle)$ is a $(p_3+1)$-cycle obtained from the $p_3$-cycle by adding an arrow which is identity. Using Proposition \ref{p cycle interpretation for extension bundle}, we have $\psi_j(E\langle\vec{x}\rangle)=E\langle\vec{x}\rangle$
for $1\leq j< p_3-l_3$, and $\psi_j(E\langle\vec{x}\rangle)=E\langle\vec{x}+\vec{x}_3\rangle$ for $p_3-l_3\leq j\leq p_3$. Hence, we get the result immediately by using
$$\psi_{j}(E\langle\vec{x}\rangle(k_3\vec{x}_3))=\psi_{j}\circ(\overline{\sigma}')^{k_3}(E\langle\vec{x}\rangle)=\overline{\sigma}^{k_3}\psi_{j+k_3}(E\langle\vec{x}\rangle).$$\end{proof}

For the case of the Auslander bundles, the following corollary holds by taking $\vec{x}=0$.

\begin{cor}\label{extension bundles 2 insertion functor}
Let $\vec{y}=\sum_{i=1}^3k_i\vec{x}_i+k\vec{c}\in \bbL(p_1,p_2,p_3)$ be its normal form. Then
\begin{equation}
\psi_j(E(\vec{y}))=
  \left\{
  \begin{array}{ll}
   E(\vec{y})& \;\;0< j< p_3-k_3\nonumber\\
  E\langle\vec{x}_3\rangle(\vec{y})& \;\;j=p_3-k_3\nonumber\\
  E(\vec{y}+\vec{x}_3)& \;\;p_3-k_3< j\leq p_3.
  \end{array}
\right.\end{equation}
\end{cor}

\section{Tilting objects in a triangulated category via recollement}

\subsection{The adjoint pair}
In this subsection, we collect some basic facts about adjoint pair. Let $\mathcal{C}$ and $\mathcal{C'}$ be arbitrary categories, and let
\begin{eqnarray}\label{adjoint pair diagram} \xymatrix@!=3pc{\mathcal{C}
\ar@<1ex>[r]^{F} &\mathcal{C'}
\ar@<1ex>[l]^{G}
}
\end{eqnarray}
be a pair of functors between them. Then $(F,G)$ is called an \emph{adjoint pair}
if there exists a natural isomorphism
$$\Hom_{\mathcal{C'}}(FC, C')\cong\Hom_{\mathcal{C}}(C, GC'),$$ which are functorial in objects $C\in\mathcal{C}$ and $C'\in\mathcal{C'}$. We also call $F$ a left adjoint of $G$ and $G$ a right adjoint of $F$.
A sequence of functors $(F,G,H)$ is called an \emph{adjoint triple} if both of $(F,G)$ and $(G,H)$
are adjoint pairs.
The following three statements are equivalent for an adjoint pair $(F,G)$ (see \cite[Prop. I.1.3]{[GZ]}):
\begin{itemize}
 \item [(1)] $FG\cong {\rm id}_{\mathcal{C'}}$;
  \item [(2)] $G$ is fully faithful;
   \item [(3)] $F$ is a quotient functor.
\end{itemize}

\subsection{The recollement and ladder}
In this subsection, we recall the definitions of the recollement and ladder.
Let $\mathcal{C},\mathcal{C}',\mathcal{C}''$ be triangulated categories.
A {\em recollement} of $\mathcal{C}$ relative to $\mathcal{C}'$ and $\mathcal{C}''$ is expressed by the following diagram
\begin{equation}
\xymatrix@C=1cm{
\mathcal{C}'\ar[rrr]|{i_*=i_!}
&&& \mathcal{C}
\ar@/_1.5pc/[lll]|{i^*}
\ar@/^1.5pc/[lll]|{i^!}
\ar[rrr]|{j^!=j^*}
&&& \mathcal{C}''
\ar@/_1.5pc/[lll]|{j_!}
\ar@/^1.5pc/[lll]|{j_*}
}\end{equation}
given by six exact functors $i^{*}$, $i_{*}=i_!$, $i^{!}$, $j_!$, $j^!=j^{*}$, $j_{*}$ satisfying the following conditions:
\begin{enumerate}
\item
$(i^\ast,i_\ast, i^!)$ and $(j_!, j^\ast, j_\ast)$ are adjoint triples;
\item  $i_\ast,\,j_\ast,\,j_!$  are fully faithful;
\item  $i^!\circ j_\ast=0$ (and thus also $j^\ast\circ i_\ast=0$ and $i^\ast\circ j_!=0$).
\item  for each $Z\in \mathcal{C},$ the units and counits of the adjunctions yield triangles
    $$\xymatrix{i_\ast i^!Z\ar[r]& Z\ar[r]& j_\ast j^\ast Z\ar[r]& i_\ast i^!Z[1]},$$
    $$\xymatrix{j_!j^\ast Z\ar[r]& Z\ar[r]& i_\ast i^\ast Z\ar[r]&j_!j^\ast Z[1]},$$
\end{enumerate}

A {\em ladder} $\mathcal L$ is a finite or infinite diagram of triangulated categories and triangle functors
\[
\xymatrix@C=1.2cm{
\mathcal{C}'
\ar@/^-3pc/[rrr]_{\vdots}|{\varphi_{j+1}}
\ar @/^3pc/[rrr]^{\vdots}|{\varphi_{j-1}}
\ar[rrr]|{\varphi_j}
&&& \mathcal{C}
\ar@/_1.5pc/[lll]|{\varphi^j}
\ar@/^1.5pc/[lll]|{\varphi^{j+1}}
\ar[rrr]|{\psi^i}
\ar@/^-3pc/[rrr]_{\vdots}|{\psi^{i+1}}
\ar @/^3pc/[rrr]^{\vdots}|{\psi^{i-1}}
&&& \mathcal{C}''.
\ar@/_1.5pc/[lll]|{\psi_{i-1}}
\ar@/^1.5pc/[lll]|{\psi_{i}}
}
\]
such that any three consecutive rows form a recollement. The rows are labelled by a subset of $\mathbb{Z}$ and multiple occurrence of the same recollement is allowed. The {\em height} of a ladder is the number of recollements contained in it (counted with multiplicities). A recollement is considered to be a ladder of height $1$.

\begin{defn}
Let $\mathcal{C}$ be a triangulated category. An object $T\in \mathcal{C}$ is called {\em tilting} if $\Hom_\mathcal{C}(T,T[n])=0$ for all nonzero integers $n\neq 0$ and $T$ generates $\mathcal{C}$ as a thick subcategory.
\end{defn}

\begin{thm}\label{tilting in recollement}
Let \begin{equation}
\xymatrix@C=1cm{
\mathcal{C}'\ar[rrr]|{i_*=i_!}
&&& \mathcal{C}
\ar@/_1.5pc/[lll]|{i^*}
\ar@/^1.5pc/[lll]|{i^!}
\ar[rrr]|{j^!=j^*}
&&& \mathcal{C}''
\ar@/_1.5pc/[lll]|{j_!}
\ar@/^1.5pc/[lll]|{j_*}
}\end{equation}
be a recollement of triangulated categories. Assume $j_{*}$ admits a right adjoint $j^\sharp$.
Assume $T'$ and $T''$ are tilting objects in $\mathcal{C}'$ and $\mathcal{C}''$, respectively. Then the following statements are equivalent:
\begin{enumerate}
\item[(1)] $T=i_*(T') \oplus j_*(T'')$ is a tilting object in $\mathcal{C}$.
\item[(2)] $\Hom_{\mathcal{C}''}(T'', j^\sharp i_*(T')[n])=0$ for all nonzero integers $n\neq 0$.
\item[(3)] $\Hom_{\mathcal{C}'}(i^*j_*(T''), T'[n])=0$ for all nonzero integers $n\neq 0$.
\end{enumerate}
\end{thm}
\begin{proof}
By assumption, $(j_{*},j^\sharp)$ is an adjoint pair, and $(i^{*}, i_{*})$ is an adjoint pair according to the definition of recollement. Then, we obtain that
$$\Hom_{\mathcal{C}''}(T'', j^\sharp i_*(T')[n])=\Hom_{\mathcal{C}}(j_*(T''), i_*T'[n])=\Hom_{\mathcal{C}'}(i^{*}j_*T'', T'[n]).$$
Hence the statement $(2)$ is equivalent to $(3)$.

(1)$\Rightarrow$(2): Assume that $T=i_*(T') \oplus j_*(T'')$ is a tilting object in $\mathcal{C}$. Since $(j_*,j^\sharp)$ is an adjoint pair, we have
$$\Hom_{\mathcal{C}''}(T'', j^\sharp i_*(T')[n])=\Hom_{\mathcal{C}}(j_*T'', i_*T'[n])=0$$
for all nonzero integers $n\neq 0$.
Hence the statement $(2)$ holds.

(2)$\Rightarrow$(1): First, we prove that $\Hom_\mathcal{C}(T,T[n])=0$ for all nonzero integers $n\neq 0$. In fact, note that $i_*$ and $j_{*}$ are fully faithful, we obtain $$\Hom_{\mathcal{C}}(i_*T', i_*T'[n])=\Hom_{\mathcal{C}'}(T', T'[n])=0,$$
$$\Hom_{\mathcal{C}}(j_*T'', j_*T''[n])=\Hom_{\mathcal{C}''}(T'', T''[n])=0$$
for all nonzero integers $n\neq 0$.
According to $(j^*,j_*)$ and $(j_*,j^\sharp)$ are adjoint pairs, therefore, for all nonzero integers $n\neq 0$, we have
$$\Hom_{\mathcal{C}}(i_*T', j_{*}T''[n])=\Hom_{\mathcal{C}}(j^* i_*(T'), T''[n])=0,$$
$$\Hom_{\mathcal{C}}(j_*T'', i_*T'[n])=\Hom_{\mathcal{C}''}(T'',j^\sharp i_*T'[n])=0.$$
It follows that $\Hom_\mathcal{C}(T,T[n])=0$ for all nonzero integers $n\neq 0$.

Next, we prove that $T$ generates $\mathcal{C}$ as a thick subcategory. That is, for any $Z\in \mathcal{C}$, we need to prove $Z\in\langle T\rangle$. According to the definition of recollement, we obtain the following triangle
$$\xymatrix{i_* i^{!}Z\ar[r]& Z\ar[r]& j_*j^* Z\ar[r]&i_* i^{!}Z[1]}.$$
Since $i^{!}Z\in \mathcal{C}'=\langle T'\rangle$ and $j^* Z\in\mathcal{C}''=\langle T''\rangle$, we obtain that $Z\in\langle T\rangle.$ Hence, $\langle T\rangle=\mathcal{C}$, that is, $T$ generates $\mathcal{C}$ as a thick subcategory. It follows that $T=i_*(T') \oplus j_*(T'')$ is a tilting object in $\mathcal{C}$.
\end{proof}

\begin{cor}
\label{tilting in recollement for add}
Let \begin{equation}
\xymatrix@C=1cm{
\mathcal{C}'\ar[rrr]|{i_*=i_!}
&&& \mathcal{C}
\ar@/_1.5pc/[lll]|{i^*}
\ar@/^1.5pc/[lll]|{i^!}
\ar[rrr]|{j^!=j^*}
&&& \mathcal{C}''
\ar@/_1.5pc/[lll]|{j_!}
\ar@/^1.5pc/[lll]|{j_*}
}\end{equation}
be a recollement of triangulated categories. Assume $j_{*}$ admits a right adjoint $j^\sharp$. Assume $T'$ and $T''$ are tilting objects in $\mathcal{C}'$ and $\mathcal{C}''$, respectively. If one of the following two conditions holds:
\begin{enumerate}
\item[(i)]
$j^\sharp i_*(T')\in\mathrm{add}(T'')$;
\item[(ii)] $i^*j_*(T'')\in\mathrm{add}(T')$.
\end{enumerate}
Then $T=i_*(T') \oplus j_*(T'')$ is a tilting object in $\mathcal{C}$.
\end{cor}

\begin{proof}
If the conditions $(i)$ or $(ii)$ holds, then we obtain that $\Hom_{\mathcal{C}''}(T'', j^\sharp i_*(T')[n])=0$ for all nonzero integers $n\neq 0$ or $\Hom_{\mathcal{C}'}(i^*j_*(T''), T'[n])=0$ for all nonzero integers $n\neq 0$, respectively. According to Theorem \ref{tilting in recollement}, we get $T=i_*(T') \oplus j_*(T'')$ is a tilting object in $\mathcal{C}$.
\end{proof}
Similar to the theorem \ref{tilting in recollement}, we have the following corollary.

\begin{cor}
Let \begin{equation}
\xymatrix@C=1cm{
\mathcal{C}'\ar[rrr]|{i_*=i_!}
&&& \mathcal{C}
\ar@/_1.5pc/[lll]|{i^*}
\ar@/^1.5pc/[lll]|{i^!}
\ar[rrr]|{j^!=j^*}
&&& \mathcal{C}''
\ar@/_1.5pc/[lll]|{j_!}
\ar@/^1.5pc/[lll]|{j_*}
}\end{equation}
be a recollement of triangulated categories. Assume $j_{!}$ admits a left adjoint $j^\sharp$.
Assume $T'$ and $T''$ are tilting objects in $\mathcal{C}'$ and $\mathcal{C}''$, respectively. Then the following statements are equivalent:
\begin{enumerate}
\item[(1)] $T=i_*(T') \oplus j_!(T'')$ is a tilting object in $\mathcal{C}$.
\item[(2)] $\Hom_{\mathcal{C}''}(j^\sharp i_*(T'), T^{''}[n])=0$ for all nonzero integers $n\neq 0$.

\item[(3)] $\Hom_{\mathcal{C}'}(T', i^!j_!(T'')[n])=0$ for all nonzero integers $n\neq 0$.
\end{enumerate}
\end{cor}

\section{Tilting objects in the stable categories of vector bundles}
In this section, we combine Theorem \ref{tilting in recollement} with the ladder and recollements on the stable category of vector bundles constructed in \cite{[R]} to give a new, unified proof of the existence of
\begin{itemize}
    \item [(i)]
tilting cuboid, as shown in \cite{[KLM1]};
\item [(ii)] tilting objects consisting of Auslander bundles, as shown in \cite{[DR]}.
\end{itemize}
This approach highlights the power of gluing techniques in recovering and contextualizing these classical results. Moreover, we also construct new tilting objects in the stable category of vector bundles over a weighted projective line.

\subsection{Ladders for the stable categories of vector bundles}

Let $\bbX_t=\bbX({\mathbf{p}}, {\boldsymbol\lambda})$ be a weighted projective line of weight type ${\mathbf{p}}=(p_1, p_2,\cdots,p_t)$ with parameter data ${\boldsymbol\lambda}=(\lambda_1,\lambda_2\cdots,\lambda_t)$.
Denote by $\vect\bbX_t$ the full subcategory of $\coh\bbX_t$ consisting of all the vector bundles.
We introduce briefly a Frobenius exact structure on $\vect\bbX_t$, for more details we refer to \cite[Section 3]{[KLM1]}.

A sequence $\xi: 0\to X'\to X\to X''\to 0$ in $\vect\bbX_t$ is called \emph{distinguished exact}
if for each line bundle $L$ the induced sequence $\Hom_{\coh\bbX_t}(L, \xi):$
$$
0\to \Hom_{\coh\bbX_t}(L, X')\to \Hom_{\coh\bbX_t}(L, X)\to \Hom_{\coh\bbX_t}(L, X'')\to 0$$ is exact. By using Serre duality, we know that $\xi$ is distinguished exact if and only if $\Hom_{\coh\bbX_t}(\xi, L)$ is exact for each line bundle $L$.
According to \cite[Proposition 3.2]{[KLM1]}, the distinguished exact sequences define a Frobenius exact structure on
$\vect\bbX_t$, such that the indecomposable projectives (resp. injectives) are exactly the line bundles. It follows that the associated stable category $\underline{\vect}\bbX_t$ obtained from $\vect\bbX_t$ by factoring out all the line bundles is a triangulated category.

For simplicity, we denote by
$\bbY:=\bbX_{t-1}=\bbX(p_1, p_2,\cdots, p_{t-1}; \lambda_1,\lambda_2, \cdots,\lambda_{t-1})$, and use the notation $\bbY(q):=\bbX(p_1, p_2,\cdots, p_{t-1},
q; \lambda_1, \lambda_2,\cdots,\lambda_t)$ for any $1\leq q\leq p:=p_t$, {where $\lambda_t$ is different from any of $\lambda_1,\lambda_2,\cdots, \lambda_{t-1}$.}
It follows that $\coh \bbY(q)=(\coh\bbY)(q; \lambda_{t})$.

According to \cite[Lemma 4.3]{[R]}, for any integer $j$, the functors $\psi^j$ and $\psi_j$ induce triangulated functors between the stable triangulated categories:
$$\bar{\psi}^j: \underline{\vect}\bbY(p) \to\underline{\vect}\bbY(p-1)\quad\text{and}\quad
\bar{\psi}_j: \underline{\vect}\bbY(p-1) \to\underline{\vect}\bbY(p).$$

The stable category of vector bundles over a weighted projective line of weight triple case is of particular interest, see \cite{[KLM1]}. From now on, we assume $\bbY(p)$ is of weight triple. In this case, $\bbY$ has weight type $(p_1, p_2)$. For any $0\leq q<p$ and any integer sequence $J_{q}=(j_1, j_2,\cdots, j_q)$
with $j_1<j_2<\cdots<j_{q}<j_1+p$, we define $\bar{\psi}^{J_{q}}$and $\bar{\psi}_{J_{q}}$ as the following compositions:
$$\bar{\psi}^{J_{q}}=\bar{\psi}^{j_1}\bar{\psi}^{j_2}\cdots \bar{\psi}^{j_q}: \underline{\vect}\bbY(p) \to\underline{\vect}\bbY(p-q)$$ and
$$\bar{\psi}_{J_{q}}=\bar{\psi}_{j_q}\cdots \bar{\psi}_{j_2}\bar{\psi}_{j_1}: \underline{\vect}\bbY(p-q) \to\underline{\vect}\bbY(p).$$

\begin{thm}\textup{\cite[Theorem 4.9]{[R]}}\label{main theorem for stab cat} Assume $\bbY$ has weight type $(p_1, p_2)$. For any $0\leq q<p$, let $J_{q}=(1,2,\cdots, q)$ and $J_{q}^c=(q+1,q+2,\cdots,p-1)$.
Then the following diagram is an infinite ladder of period $p$:
$$\xymatrix@C=0.5cm{
  \underline{\vect}\bbY(q+1)
\ar@/_3pc/[rrr]_{\vdots}|{\bar{\psi}_{J_{q}^c+\underline{1}}}
\ar@/^3pc/[rrr]^{\vdots}|{\bar{\psi}_{J_{q}^c-\underline{1}}}
\ar[rrr]|{\bar{\psi}_{{J_{q}^c}}}
&&& \underline{\vect}\bbY(p)
\ar@/_1.5pc/[lll]|{\bar{\psi}^{{J_{q}^c}}}
\ar@/^1.5pc/[lll]|{\bar{\psi}^{J_q^c+{\underline{1}}}}
\ar[rrr]|{\bar{\psi}^{J_{q}}}
\ar@/_3pc/[rrr]_{\vdots}|{\bar{\psi}^{{J_{q}}+{\underline{1}}}}
\ar@/^3pc/[rrr]^{\vdots}|{\bar{\psi}^{{J_{q}}-{\underline{1}}}}
&&& \underline{\vect}\bbY(p-q).
\ar@/_1.5pc/[lll]|{\bar{\psi}_{{J_{q}}-{\underline{1}}}}
\ar@/^1.5pc/[lll]|{\bar{\psi}_{J_{q}}}
}$$
\end{thm}
Here, the sequence $\underline{1}$ has the same size as $J_{q}$ and has entry 1's everywhere, and $J_{q}\pm \underline{1}$ are defined via componentwise addition.

Theorem \ref{tilting in recollement} and Theorem \ref{main theorem for stab cat} together give an effective recipe for producing tilting objects in the stable category $\underline{\vect}\bbX(p_1, p_2, p_3)$, because all the functors in the ladder of Theorem \ref{main theorem for stab cat} are explicitly described. Using this machinery we recover both the tilting cuboid of \cite{[KLM1]} and the tilting objects formed by Auslander bundles as shown in \cite{[DR]}.
Moreover, we construct new tilting objects in the stable category $\underline{\vect}\bbX(2, p_2, p_3)$.

\subsection{New proof for tilting cuboid}
In this subsection, we give a new proof of the tilting cuboid from \cite{[KLM1]}, showing how it is glued from lines and rectangles.

\begin{prop}\label{CubTilting}\textup{\cite[Theorem 6.1]{[KLM1]}}
Assume $\vec{\delta}=\sum_{i=1}^{3}(p_i-2)\vec{x}_i\in \bbL(p_1,p_2,p_3)$. Then
$T_{{\rm cub}}=\oplus_{0\leq \vec{x}\leq \vec{\delta}} E\langle \vec{x}\rangle$ is a tilting object in $\underline{\vect}\bbX(p_1, p_2, p_3).$ \end{prop}

\begin{proof}
Denote by $$T_{(a_1,a_2,a_3)}=\bigoplus_{i=0}^{a_1-2}\bigoplus_{j=0}^{a_2-2}\bigoplus_{k=0}^{a_3-2} E\langle i\vec{x}_1+j\vec{x}_2+k\vec{x}_3\rangle.$$
Then $T_{{\rm cub}}=T_{(p_1,p_2,p_3)}$.

First, we prove $T_{(2,2,p_3)}$ is a tilting object in $\underline{\vect}\bbX(2,2,p_3)$ by induction on $p_3$. It is obvious that $T_{(2,2,2)}=E$ is a tilting object in $\underline{\vect}\bbX(2,2,2)$. Assume $T_{(2,2,p_3-1)}$ is a tilting object in $\underline{\vect}\bbX(2,2,p_3-1)$. We consider the following ladder, where ${J_{1}^c}=(2,3,\cdots,p_3-1)$.
\begin{figure}[ht]
    $$\xymatrix@C=0.5cm{
  \underline{\vect}\bbX(2,2,2)
\ar@/_3pc/[rrr]_{\vdots}|{\bar{\psi}_{J_{1}^c+\underline{1}}}
\ar@/^3pc/[rrr]^{\vdots}|{\bar{\psi}_{J_{1}^c-\underline{1}}}
\ar[rrr]|{\bar{\psi}_{J_{1}^c}}
&&& \underline{\vect}\bbX(2,2,p_3)
\ar@/_1.5pc/[lll]|{\bar{\psi}^{J_{1}^c}}
\ar@/^1.5pc/[lll]|{\bar{\psi}^{J_{1}^c+\underline{1}}}
\ar[rrr]|{\bar{\psi}^{1}}
\ar@/_3pc/[rrr]_{\vdots}|{\bar{\psi}^{2}}
\ar@/^3pc/[rrr]^{\vdots}|{\bar{\psi}^{0}}
&&& \underline{\vect}\bbX(2,2,p_3-1)
\ar@/_1.5pc/[lll]|{\bar{\psi}_{0}}
\ar@/^1.5pc/[lll]|{\bar{\psi}_{1}}
}$$
\caption{Ladder for $\underline{\vect}\bbX(2,2,p_3)$}
    \label{LadderFor22p}
\end{figure}

 \noindent By Proposition \ref{extension bundles insertion functor}, we have $${\bar{\psi}_{J_{1}^c}}(T_{(2,2,2)})=\bar{\psi}_{p_3-1}\bar{\psi}_{p_3-2}\cdots\bar{\psi}_{2}(E)=E\langle (p_3-2)\vec{x}_3\rangle,$$ and $$\bar{\psi}_1(T_{(2,2,p_3-1)})=\oplus_{0\leq k\leq p_3-3} E\langle k\vec{x}_3\rangle.$$

By Proposition \ref{element reduction functor}, we obtain
$$\bar{\psi}^{2}{\bar{\psi}_{J_{1}^c}}(T_{(2,2,2)})=\bar{\psi}^{2}(E\langle (p_3-2)\vec{x}_3\rangle)=E\langle (p_3-3)\vec{x}_3\rangle,$$ which is a direct summand of $T_{(2,2,p_3-1)}$. According to Corollary
\ref{tilting in recollement for add}, we obtain that $$T_{(2,2,p_3)}={\bar{\psi}_{J_{1}^c}}(T_{(2,2,2)})\oplus\bar{\psi}_1(T_{(2,2,p_3-1)})$$ is a tilting object in $\underline{\vect}\bbX(2,2,p_3)$.

Next, we consider the following ladder, see Figure \ref{LadderFor2p1p2}.
\begin{figure}[ht]
    $$\xymatrix@C=0.5cm{
  \underline{\vect}\bbX(2,p_2,2)
\ar@/_3pc/[rrr]_{\vdots}|{\bar{\psi}_{J_{1}^c+\underline{1}}}
\ar@/^3pc/[rrr]^{\vdots}|{\bar{\psi}_{J_{1}^c-\underline{1}}}
\ar[rrr]|{\bar{\psi}_{J_{1}^c}}
&&& \underline{\vect}\bbX(2,p_2,p_3)
\ar@/_1.5pc/[lll]|{\bar{\psi}^{J_{1}^c}}
\ar@/^1.5pc/[lll]|{\bar{\psi}^{J_{1}^c+\underline{1}}}
\ar[rrr]|{\bar{\psi}^{1}}
\ar@/_3pc/[rrr]_{\vdots}|{\bar{\psi}^{2}}
\ar@/^3pc/[rrr]^{\vdots}|{\bar{\psi}^{0}}
&&& \underline{\vect}\bbX(2,p_2,p_3-1)
\ar@/_1.5pc/[lll]|{\bar{\psi}_{0}}
\ar@/^1.5pc/[lll]|{\bar{\psi}_{1}}
}$$
\caption{Ladder for $\underline{\vect}\bbX(2,p_2,p_3)$}
    \label{LadderFor2p1p2}
\end{figure}

By Propositions \ref{element reduction functor} and \ref{extension bundles insertion functor}, we observe that the coefficients of $\vec{x}_1$ and $\vec{x}_2$ in $\vec{x}$ and $\vec{y}$ do not affect the process of element change under the reduction functor and the insertion functor. According to the first step, we get that $\bigoplus_{i=0}^{p_2-2} E\langle i\vec{x}_2\rangle$ is a tilting object in $\underline{\vect}\bbX(2,p_2,2)$. By Proposition \ref{extension bundles insertion functor}, we obtain that
$${\bar{\psi}_{J_{1}^c}}(T_{(2,p_2,2)})=\bigoplus_{i=0}^{p_2-2} E\langle i\vec{x}_2+(p_3-2)\vec{x}_3\rangle,\;\;\bar{\psi}_1(T_{(2,p_2,p_3-1)})=\bigoplus_{j=0}^{p_2-2}\bigoplus_{k=0}^{p_3-3} E\langle j\vec{x}_2+k\vec{x}_3\rangle$$
hold for the ladder, and by Proposition \ref{element reduction functor},
$$\bar{\psi}^{2}\bar{\psi}_{J_{1}^c}(T_{(2,p_2,2)})=\bar{\psi}^{2}(\bigoplus_{i=0}^{p_2-2} E\langle i\vec{x}_2+(p_3-2)\vec{x}_3\rangle)=\bigoplus_{i=0}^{p_2-2} E\langle i\vec{x}_2+(p_3-3)\vec{x}_3\rangle.$$
It is a direct summand of $T_{(2,p_2,p_3-1)}$. By induction, $T_{(2,p_2,p_3-1)}$ is a tilting object in the category $\underline{\vect}\bbX(2,p_2,p_3-1)$.
Hence, by Corollary \ref{tilting in recollement for add}, we obtain that $T_{(2,p_2,p_3)}={\bar{\psi}_{J_{1}^c}}(T_{(2,p_2,2)})\oplus\bar{\psi}_1(T_{(2,p_2,p_3-1)})$ is a tilting object in $\underline{\vect}\bbX(2,p_2,p_3)$.

Finally, we consider the following ladder, see Figure \ref{LadderForp1p2p3}.

\begin{figure}[ht]
$$
\xymatrix@C=0.5cm{
  \underline{\vect}\bbX(p_1,p_2,2)
\ar@/_3pc/[rrr]_{\vdots}|{\bar{\psi}_{J_{1}^c+\underline{1}}}
\ar@/^3pc/[rrr]^{\vdots}|{\bar{\psi}_{J_{1}^c-\underline{1}}}
\ar[rrr]|{\bar{\psi}_{J_{1}^c}}
&&& \underline{\vect}\bbX(p_1,p_2,p_3)
\ar@/_1.5pc/[lll]|{\bar{\psi}^{J_{1}^c}}
\ar@/^1.5pc/[lll]|{\bar{\psi}^{J_{1}^c+\underline{1}}}
\ar[rrr]|{\bar{\psi^{1}}}
\ar@/_3pc/[rrr]_{\vdots}|{\bar{\psi^{2}}}
\ar@/^3pc/[rrr]^{\vdots}|{\bar{\psi^{0}}}
&&& \underline{\vect}\bbX(p_1,p_2,p_3-1)
\ar@/_1.5pc/[lll]|{\bar{\psi_{0}}}
\ar@/^1.5pc/[lll]|{\bar{\psi_{1}}}
}
$$
\caption{Ladder for $\underline{\vect}\bbX(p_1,p_2,p_3)$}
    \label{LadderForp1p2p3}
\end{figure}
According to the second step, we have $\bigoplus_{i=0}^{p_1-2}\bigoplus_{j=0}^{p_2-2} E\langle i\vec{x}_1+j\vec{x}_2\rangle$ is a tilting object in $\underline{\vect}\bbX(p_1,p_2,p_3)$. By Proposition \ref{extension bundles insertion functor}, we obtain that
\begin{align*}
    {\bar{\psi}_{J_{1}^c}}(T_{(p_1,p_2,2)})=\bigoplus_{i=0}^{p_1-2}\bigoplus_{j=0}^{p_2-2} E\langle i\vec{x}_1+j\vec{x}_2+(p_3-2)\vec{x}_3\rangle,\\
    \bar{\psi}_1(T_{(p_1,p_2,p_3-1)})=\bigoplus_{i=0}^{p_1-2}\bigoplus_{j=0}^{p_2-2}\bigoplus_{k=0}^{p_3-3} E\langle i\vec{x}_1+j\vec{x}_2+k\vec{x}_3\rangle
\end{align*}
hold, and by Proposition \ref{element reduction functor},
$$\bar{\psi}^{2}\bar{\psi}_{J_{1}^c}(T_{(p_1,p_2,2)})=\bar{\psi}^{2}(\bigoplus_{i=0}^{p_1-2}\bigoplus_{j=0}^{p_2-2} E\langle i\vec{x}_1+j\vec{x}_2+(p_3-2)\vec{x}_3\rangle)=\bigoplus_{i=0}^{p_1-2}\bigoplus_{j=0}^{p_2-2} E\langle i\vec{x}_1+j\vec{x}_2+(p_3-3)\vec{x}_3\rangle,$$
which is a direct summand of $T_{(p_1,p_2,p_3-1)}$. By induction, $T_{(p_1,p_2,p_3-1)}$ is a tilting object in $\underline{\vect}\bbX(p_1,p_2,p_3-1)$.
Hence, by Corollary \ref{tilting in recollement for add}, we obtain that $T_{(p_1,p_2,p_3)}={\bar{\psi}_{J_{1}^c}}(T_{(p_1,p_2,2)})\oplus\bar{\psi}_1(T_{(p_1,p_2,p_3-1)})$ is a tilting object in $\underline{\vect}\bbX(p_1,p_2,p_3)$. Therefore, $T_{{\rm cub}}=T_{(p_1,p_2,p_3)}$ is a tilting object in $\underline{\vect}\bbX(p_1, p_2, p_3)$.
\end{proof}

\subsection{New proof for tilting objects consisting of Auslander bundles}
In this subsection, we provide an essential proof for tilting objects consisting of Auslander bundles \cite{[DR]}.

Denote by $\overline{x}_i=\vec{x}_i+\vec{\omega}$ for $i=1,2,3$. In $\underline{\vect}\bbX(2, p_2, p_3)$, there are two tilting objects  consisting of Auslander bundles $\bigoplus\limits_{i=0}^{p_3-2}\bigoplus\limits_{j=0}^{p_2-2}E(i\overline{x}_2+j\overline{x}_3)$ and $\bigoplus\limits_{i=0}^{p_3-2}\bigoplus\limits_{j=0}^{p_2-2}E(i\overline{x}_1+j\overline{x}_3)$
established in \cite{[DR]}. Here, we present an alternative proof of this fact by employing Theorem \ref{tilting in recollement} and Theorem \ref{main theorem for stab cat}.

\begin{prop}\label{DRTilting}
There are two tilting object in $\underline{\vect}\bbX(2, p_2, p_3)$
\begin{itemize}
  \item[(1)] $T_1=\oplus_{0\leq i\leq p_3-2, 0\leq j\leq p_2-2} E(i\overline{x}_2+j\overline{x}_3)$;
  \item[(2)] $T_2=\oplus_{0\leq i\leq p_3-2, 0\leq j\leq p_2-2} E(i\overline{x}_1+j\overline{x}_3)$.
   \end{itemize}
      \end{prop}

\begin{proof}
First, we prove that $T_{(2,2,p)}=\oplus_{0\leq i\leq p-2}E(i\overline{x}_2+(p-2)\vec{x}_3)$ is a tilting object in $\underline{\vect}\bbX(2,2,p)$ by induction on $p$.
In fact, it is clear that $E(\vec{x})$ is a tilting object in $\underline{\vect}\bbX(2, 2, 2)$ for each $\vec{x}\in\mathbb{L}(2,2,2)$. For $p\geq3$, we consider the ladder in Figure \ref{LadderFor22p}. By Proposition \ref{extension bundles insertion functor}, we obtain
$$\bar{\psi}_1(T_{(2,2,p-1)})=\oplus_{0\leq i\leq p-3} E(i\overline{x}_2+(p-3)\vec{x}_3).$$
In $\underline{\vect}\bbX(2,2,2)$, we take $T'_p:=E(p\vec{x}_1-\vec{x}_2-\vec{c})$ as the tilting object. Then
$${\bar{\psi}_{J_{1}^c}}(T'_p)=\bar{\psi}_{p-1}\bar{\psi}_{p-2}\cdots\bar{\psi}_{2}(T'_p)=E\langle(p-2)\vec{x}_3\rangle(p\vec{x}_1-\vec{x}_2-\vec{c})=E\big((p-2)\overline{x}_2+(p-3)\vec{x}_3\big).$$
By Corollary \ref{extension bundles 2 reduction functor}, we have
\begin{align*}
  &\Hom_{\underline{\vect}\bbX(2,2,p-1)}(T_{(2,2,p-1)}, \bar{\psi}^{2}\bar{\psi}_{J_{1}^c}(T'_p)[n])\\
  =&\bigoplus_{i=0}^{p-3}\Hom_{\underline{\vect}\bbX(2,2,p-1)}\big(E(i\overline{x}_2+(p-3)\vec{x}_3),E((p-2)\vec{x}_1-2\vec{x}_3)[n]\big).
\end{align*}
According to \cite[Corollary 4.14 and Proposition 6.8]{[KLM1]}, this hom-set is non-zero if and only if
$$[(p-2+n)\vec{x}_1-2\vec{x}_3]-[i\overline{x}_2+(p-3)\vec{x}_3]\in\{0,\overline{x}_1,\overline{x}_2,\overline{x}_3\}$$ holds for some $0\leq i\leq p-3$. It is ipmossible for $n\neq 0$. Therefore, the hom-set is zero for all $n\neq0$. Combining with Theorem \ref{tilting in recollement}, one has $$T_{(2,2,p)}={\bar{\psi}_{J_{1}^c}}(T'_p)\oplus\bar{\psi}_1(T_{(2,2,p-1)})$$ is a tilting object by recursion. Hence, $\oplus_{0\leq i\leq p-2}E(i\overline{x}_2)$ is also a tilting object in $\underline{\vect}\bbX(2,2,p)$.

Next, we consider the ladder for the weight $(2,p_2,p_3)$ as in Figure \ref{LadderFor2p1p2}. As we proved above, $\oplus_{j=0}^{p_2-2}E\big(j\overline{x}_3+(p_3-3)\vec{x}_1\big)$ is a tilting object in $\underline{\vect}\bbX(2,p_2,2)$. Note that
$${\bar{\psi}_{J_{1}^c}}[E\big(j\overline{x}_3+(p_3-3)\vec{x}_1\big)]=E\langle(p-2)\vec{x}_3\rangle\big(j\overline{x}_3+(p_3-3)\vec{x}_1\big)=E\big(j\overline{x}_3+(p_3-2)\overline{x}_2+(p-3)\vec{x}_3\big).$$
In $\underline{\vect}\bbX(2,p_2,p_3-1)$, We take $\oplus_{i=0}^{p_3-3}\oplus_{j=0}^{p_2-2}E\big(i\overline{x}_2+j\overline{x}_3+(p_3-3)\vec{x}_3\big)$ as the tilting object which preserves fixed under the functor $\bar{\psi}_1$. We also can check that the hom-set
$$\Hom_{\underline{\vect}\bbX(2,2,p-1)}\big(\bigoplus_{j=0}^{p_2-2}E(j\overline{x}_3+(p_3-3)\vec{x}_1), \bar{\psi}^{2}\bar{\psi}_{J_{1}^c}(\bigoplus_{i=0}^{p_3-3}\bigoplus_{j=0}^{p_2-2}E(i\overline{x}_2+j\overline{x}_3+(p_3-3)\vec{x}_3))[n]\big)$$
is zero for all $n\neq0$ by \cite[Corollary 4.14 and Proposition 6.8]{[KLM1]}. Hence, $\oplus_{0\leq i\leq p_3-2, 0\leq j\leq p_2-2} E(i\overline{x}_2+j\overline{x}_3)$ is a tilting object in $\underline{\vect}\bbX(2, p_2, p_3)$.

Finally, by the similar way, we can prove that $\oplus_{0\leq i\leq p_3-2, 0\leq j\leq p_2-2} E(i\overline{x}_1+j\overline{x}_3)$ is a tilting object in $\underline{\vect}\bbX(2,p_2,p_3)$
\end{proof}

\subsection{New classes of tilting objects in the stable categories of vector bundles}

Based on Theorem \ref{tilting in recollement} and Theorem \ref{main theorem for stab cat} , we construct some ``new'' tilting objects in $\underline{\vect}\bbX(2, p_2, p_3)$.

\begin{thm}\label{newtiltingobject}
Let $q$ be an integer with $1\leq q\leq p_3-2$. For $k=1,2$, denote by
$$T_{1k}=\bigoplus_{i=0}^{q-1}\bigoplus_{j=0}^{p_2-2}E(i\overline{x}_k+j\overline{x}_3-\vec{x}_3+\vec{c})$$
and
$$T_{2k}=\bigoplus_{j=0}^{p_2-2}[E\langle q\vec{x}_3\rangle\big(j\overline{x}_3-(q+1)\vec{x}_3+\vec{c}\big)\bigoplus\bigoplus_{i=1}^{p_3-q-2}E\big(i\overline{x}_k+j\overline{x}_3-(q+1)\vec{x}_3+\vec{c}\big)]$$
in $\underline{\vect}\bbX(2, p_2, p_3)$. Then $T_{1i}\oplus T_{2j}, 1\leq i, j\leq 2$ are tilting objects in $\underline{\vect}\bbX(2, p_2, p_3)$.
\end{thm}

\begin{proof}
Let $$T'_{1k}=\bigoplus_{i=0}^{q-1}\bigoplus_{j=0}^{p_2-2}E(i\overline{x}_k+j\overline{x}_3+q\vec{x}_3)
\;\;\textup{and}\;\;
T'_{2k}=\bigoplus_{i=0}^{p_3-q-2}\bigoplus_{j=0}^{p_2-2}E\big(i\overline{x}_k+j\overline{x}_3+(p_3-q-1)\vec{x}_3\big).$$
By Proposition \ref{DRTilting}, we know that both $T'_{11}$ and $T'_{12}$ (resp. $T'_{21}$ and $T'_{22}$) are tilting object in $\underline{\vect}\bbX(2,p_2,q+1)$ (resp. $\underline{\vect}\bbX(2,p_2,p_3-q)$).

Denote by $J_q=(1,2,\cdots,q)$. According to Theorem \ref{main theorem for stab cat}, we have the following ladder.
\begin{figure}[ht]
$$\xymatrix@C=0.5cm{
  \underline{\vect}\bbX(2,p_2,q+1)
\ar@/_3pc/[rrr]_{\vdots}|{\bar{\psi}_{J_{q}^c+\underline{1}}}
\ar@/^3pc/[rrr]^{\vdots}|{\bar{\psi}_{J_{q}^c-\underline{1}}}
\ar[rrr]|{\bar{\psi}_{{J_{q}^c}}}
&&& \underline{\vect}\bbX(2,p_2,p_3)
\ar@/_1.5pc/[lll]|{\bar{\psi}^{{J_{q}^c}}}
\ar@/^1.5pc/[lll]|{\bar{\psi}^{J_q^c+{\underline{1}}}}
\ar[rrr]|{\bar{\psi}^{J_{q}}}
\ar@/_3pc/[rrr]_{\vdots}|{\bar{\psi}^{{J_{q}}+{\underline{1}}}}
\ar@/^3pc/[rrr]^{\vdots}|{\bar{\psi}^{{J_{q}}-{\underline{1}}}}
&&& \underline{\vect}\bbX(2,p_2,p_3-q).
\ar@/_1.5pc/[lll]|{\bar{\psi}_{{J_{q}}-{\underline{1}}}}
\ar@/^1.5pc/[lll]|{\bar{\psi}_{J_{q}}}
}$$
\end{figure}

By Proposition \ref{extension bundles insertion functor}, we have
$$\bar{\psi}_{J_q^c}(T'_{1k})=T_{1k},\;\;\bar{\psi}_{J_q}(T'_{2k})=T_{2k}$$
for $k=1,2$. Using the result of Hom-sets between extension bundles in \cite{[KLM1]}, it is easy to check that
$$\Hom_{\underline{\vect}\bbX(2,p_2,p_3-1)}(T'_{2s}, \bar{\psi}^{J_q+\underline{1}}{\bar{\psi}_{J_{q}^c}}(T^{\prime}_{1t})[n])=0$$
for all $s,t\in\{1,2\}$ and all $n\neq 0$ case by case. Combining with Theorem \ref{tilting in recollement}, we get ${\bar{\psi}_{J_{q}^c}}(T^{\prime}_{1s})\oplus {\bar{\psi}_{J_{q}}}(T'_{2t})$ is a tilting object in $\underline{\vect}\bbX(2, p_2, p_3)$ for all $s,t\in\{1,2\}$.
\end{proof}

The following example, we construct the endomorphism algebra for a tilting object in Theorem \ref{newtiltingobject}.

\begin{exm}
In Proposition \ref{newtiltingobject}, we take $p_2=3$, $p_3=4$ and $q=1$. Then we have
$$[T_{11}\oplus T_{22}](-2\vec{x}_3)=E(\vec{x}_3)\oplus E(\overline{x}_3+\vec{x}_3)\oplus E(\overline{x}_1)\oplus E(\overline{x}_2)\oplus E\langle\vec{x}_3\rangle\oplus  E\langle\vec{x}_3\rangle(\overline{x}_3)$$
is a tilting object in $\underline{\vect}\bbX(2,3,4)$ whose endomorphism algebra is the quiver algebra associated to the quiver as following.
$$\xymatrix{E(\vec{x}_3)\ar[d]&E\langle\vec{x}_3\rangle\ar[d]\ar[r]\ar[l]&E(\overline{x}_2)\ar[d]\\
E(\vec{x}_3+\overline{x}_3)&E\langle\vec{x}_3\rangle(\overline{x}_3)\ar[r]\ar[l]&E(\overline{x}_1)}$$
\end{exm}

\noindent {\bf Acknowledgements.}\quad
First and foremost, we extend our sincere gratitude to the Tianyuan Mathematics Center of Southeast China for its gracious hospitality and support. The first author was supported by the National Natural Science Foundation of China (Grant No. 12301054). The second author was supported by the National Natural Science Foundation of China (Grant No. 12501046). The authors are grateful to Shiquan Ruan for his continuous discussions and valuable comments.

\bibliographystyle{amsplain}

\begin{thebibliography}{20}

\bibitem{[AKL]}
L.~Angeleri~H$\ddot{u}$gel, S.~Koenig and Q.~Liu. \newblock Recollements and tilting objects.
\newblock {\em J. Pure
Appl. Algebra} 215: 420¨C438, 2011.

\bibitem{[AKLY]}
L.~Angeleri~H$\ddot{u}$gel, S.~Koenig, Q.~Liu, and D.~Yang.
\newblock Ladders and simplicity of derived module categories.
\newblock {\em J. Algebra}, 472:15--66, 2017.

\bibitem{[BDS]}
P.~Balmer, I.~Dell'Ambrogio, and B.~Sanders.
\newblock Grothendieck-{N}eeman duality and the {W}irthm$\ddot{u}$ller isomorphism.
\newblock {\em Compos. Math.}, 152(8):1740--1776, 2016.

\bibitem{[BBD]}
A.~A. Be\u{\i}linson, J.~Bernstein, and P.~Deligne.
\newblock Faisceaux pervers.
\newblock In {\em Analysis and topology on singular spaces, {I} ({L}uminy, 1981)}, volume 100 of {\em Ast\'erisque}, pages 5--171. Soc. Math. France, Paris, 1982.

\bibitem{[BGS]}
A.~A. Be\u{\i}linson, V.~A. Ginsburg, and V.~V. Schechtman.
\newblock Koszul duality.
\newblock {\em J. Geom. Phys.}, 5(3):317--350, 1988.

\bibitem{[Bo]} M.~ V. Bondarko.
\newblock Weight structures vs. t-structures; weight ?ltrations, spectral
sequences, and complexes (for motives and in general). \newblock {\em J. K-Theory} 6 (2010), 387--
504.

\bibitem{CLLR}
J.~Chen, Y.~Lin, P.~Liu, and S.~Ruan.
\newblock Tilting objects on tubular weighted projective lines: a cluster tilting approach.
\newblock {\em Sci. China Math.}, 64(4):691--710, 2021.

\bibitem{CLR2019}
J.~Chen, Y.~Lin, and S.~Ruan.
\newblock On tubular tilting objects in the stable category of vector bundles.
\newblock {\em Acta Math. Sin. (Engl. Ser.)}, 35(4):494--512, 2019.

\bibitem{[C]}
X.-W. Chen.
\newblock A recollement of vector bundles.
\newblock {\em Bull. Lond. Math. Soc.}, 44(2):271--284, 2012.

\bibitem{[CK]}
X.-W. Chen and H.~Krause.
\newblock Expansions of abelian categories.
\newblock {\em J. Pure Appl. Algebra}, 215(12):2873--2883, 2011.

\bibitem{[CPS2]}
E.~Cline, B.~Parshall, and L.~Scott.
\newblock Algebraic stratification in representation categories.
\newblock {\em J. Algebra}, 117(2):504--521, 1988.

\bibitem{[CPS1]}
E.~Cline, B.~Parshall, and L.~Scott.
\newblock Finite-dimensional algebras and highest weight categories.
\newblock {\em J. Reine Angew. Math.}, 391:85--99, 1988.

\bibitem{[DLR]}
Q. Dong, Y. Lin and S. Ruan.
\newblock Derived equivalences between one-branch extension algebras of ``rectangles''.
\newblock {\em Acta Math. Sin. (Engl. Ser.)}£¬ to appear.

\bibitem{[DR]}
Q. Dong and S. Ruan.
\newblock On two open questions for extension bundles. \newblock {\em J. Algebra}, 662:407--430, 2025.

\bibitem{[EP]}
W. Ebeling and D. Ploog.
\newblock McKay correspondence for the Poincar\'{e} series of Kleinian
  and Fuchsian singularities.
\newblock {\em Math. Ann.} 347 (2010), no. 3, 689--702.

\bibitem{[ET]}
W. Ebeling and A. Takahashi.
\newblock Mirror symmetry between orbifold curves and cusp singularities with
  group action.
\newblock {\em Int. Math. Res. Not. IMRN} 2013, no. 10, 2240--2270.

\bibitem{[GZ]}
P.~Gabriel and M.~Zisman.
\newblock {\em Calculus of fractions and homotopy theory}, volume Band 35 of {\em Ergebnisse der Mathematik und ihrer Grenzgebiete [Results in Mathematics and Related Areas]}.
\newblock Springer-Verlag New York, Inc., New York, 1967.

\bibitem{[GKP]}
N.~Gao, S.~Koenig, and C.~Psaroudakis.
\newblock Recollements of abelian categories and ideals in heredity chains---a recursive approach to quasi-hereditary algebras.
\newblock {\em Proc. Amer. Math. Soc.}, 147(11):4625--4637, 2019.

\bibitem{[GP]}
N.~Gao and C.~Psaroudakis.
\newblock Ladders of compactly generated triangulated categories and preprojective algebras.
\newblock {\em Appl. Categ. Structures}, 26(4):657--679, 2018.

\bibitem{[GL1]}
W.~Geigle and H.~Lenzing.
\newblock A class of weighted projective curves arising in representation theory of finite-dimensional algebras.
\newblock In {\em Singularities, representation of algebras, and vector bundles ({L}ambrecht, 1985)}, volume 1273 of {\em Lecture Notes in Math.}, pages 265--297. Springer, Berlin, 1987.

\bibitem{[Hap1988]}
D.~Happel.
\newblock {\em Triangulated categories in the representation theory of finite-dimensional algebras}, volume 119 of {\em London Mathematical Society Lecture Note Series}.
\newblock Cambridge University Press, Cambridge, 1988.

\bibitem{Hubery}
A. Hubery.
\newblock Coherent sheaves on weighted projective lines via periodic functors, part 1 and 2.
\newblock Seminar talks in Bielefeld.

\bibitem{[J]} P.
\newblock J{\o}rgensen, Reflecting recollements.
\newblock {\em Osaka J. Math.}, 47 (1), 209--213 (2010).

\bibitem{Krause2017HWC}
H.~Krause.
\newblock Highest weight categories and recollements.
\newblock {\em Ann. Inst. Fourier (Grenoble)}, 67(6):2679--2701, 2017.

\bibitem{[Kussin]}
D.~Kussin.
\newblock Noncommutative curves of genus zero: related to finite dimensional algebras.
\newblock {\em Mem. Amer. Math. Soc.}, 201(942):x+128, 2009.

\bibitem{[K]}
D.~Kussin.
\newblock Weighted noncommutative regular projective curves.
\newblock {\em J. Noncommut. Geom.}, 10(4):1465--1540, 2016.

\bibitem{[KLM2]}
D.~Kussin, H.~Lenzing, and H.~Meltzer.
\newblock Nilpotent operators and weighted projective lines.
\newblock {\em J. Reine Angew. Math.}, 685:33--71, 2013.

\bibitem{[KLM1]}
D.~Kussin, H.~Lenzing, and H.~Meltzer.
\newblock Triangle singularities, {ADE}-chains, and weighted projective lines.
\newblock {\em Adv. Math.}, 237:194--251, 2013.

\bibitem{[KM]}
D.~Kussin and H.~Meltzer.
\newblock The braid group action for exceptional curves.
\newblock {\em Arch. Math. (Basel)}, 79(5):335--344, 2002.

\bibitem{[La]}
S. Ladkani.
\newblock Derived equivalences of triangular matrix rings arising from extensions of tilting modules.
\newblock {\em  Algebr. Represent. Theory} 14 (2011), 57--74.

\bibitem{[L3]}
H.~Lenzing.
\newblock Hereditary {N}oetherian categories with a tilting complex.
\newblock {\em Proc. Amer. Math. Soc.}, 125(7):1893--1901, 1997.

\bibitem{[L]}
H.~Lenzing.
\newblock Representations of finite-dimensional algebras and singularity theory.
\newblock In {\em Trends in ring theory ({M}iskolc, 1996)}, volume~22 of {\em CMS Conf. Proc.}, pages 71--97. Amer. Math. Soc., Providence, RI, 1998.

\bibitem{[L1]}
H.~Lenzing.
\newblock Weighted projective lines and applications.
\newblock In {\em Representations of algebras and related topics}, EMS Ser. Congr. Rep., pages 153--187. Eur. Math. Soc., Z\"{u}rich, 2011.

\bibitem{[Ld]}
H.~Lenzing and J.~A. de~la Pe\~{n}a.
\newblock Concealed-canonical algebras and separating tubular families.
\newblock {\em Proc. London Math. Soc. (3)}, 78(3):513--540, 1999.

\bibitem{[LMR]}
H.~Lenzing, H.~Meltzer, and S.~Ruan.
\newblock Nakayama algebras and {F}uchsian singularities.
\newblock {\em Algebr. Represent. Theory}, 27(1):815--846, 2024.

\bibitem{[LR]}
H.~Lenzing and I.~Reiten.
\newblock Hereditary {N}oetherian categories of positive {E}uler characteristic.
\newblock {\em Math. Z.}, 254(1):133--171, 2006.

\bibitem{[LiuLu]}
P.~Liu and M.~Lu.
\newblock Recollements of singularity categories and monomorphism categories.
\newblock {\em Comm. Algebra}, 43(6):2443--2456, 2015.

\bibitem{[LV]}
Q.~Liu and J.~Vit\'oria.
\newblock t-structures via recollements for piecewise hereditary algebras.
\newblock {\em J. Pure Appl. Algebra} 216 (2012), 837--849.

\bibitem{[LVY]}
Q.~Liu, J.~Vit\'oria, and D.~Yang.
\newblock Gluing silting objects.
\newblock {\em Nagoya Math. J.}, 216:117--151, 2014.

\bibitem{MH}
X.~Ma and Z.~Huang.
\newblock Torsion pairs in recollements of abelian categories.
\newblock {\em Front. Math. China}, 13(4):875--892, 2018.

\bibitem{[MZ]}
X.~Ma and T.~Zhao.
\newblock Recollements and tilting modules.
\newblock {\em Comm. Algebra}, 48(12):5163--5175, 2020.

\bibitem{[P]}
C.~Psaroudakis.
\newblock A representation-theoretic approach to recollements of abelian categories.
\newblock In {\em Surveys in representation theory of algebras}, volume 716 of {\em Contemp. Math.}, pages 67--154. Amer. Math. Soc., [Providence], RI, [2018] \copyright 2018.

\bibitem{[Rin]}
C.~M. Ringel.
\newblock {\em Tame algebras and integral quadratic forms}, volume 1099 of {\em Lecture Notes in Mathematics}.
\newblock Springer-Verlag, Berlin, 1984.

\bibitem{[RS]}
C.~M. Ringel and M.~Schmidmeier.
\newblock Invariant subspaces of nilpotent linear operators. {I}.
\newblock {\em J. Reine Angew. Math.}, 614:1--52, 2008.

\bibitem{[R]}
S.~Ruan.
\newblock Recollements and ladders for weighted projective lines.
\newblock {\em J. Algebra}, 578:213--240, 2021.

\bibitem{[Simson2015]}
D.~Simson.
\newblock Tame-wild dichotomy of {B}irkhoff type problems for nilpotent linear
  operators.
\newblock {\em J. Algebra}, 424:254--293, 2015.
\end{thebibliography}

\vskip 5pt
\noindent {\tiny  \noindent Qiang Dong, Hongxia Zhang\\
School of Mathematics and Statistics, \\
 Fujian Normal University, Fuzhou, 350117,
Fujian, PR China.\\
E-mail: dongqiang@fjnu.edu.cn, hxzhangxmu@163.com\\}
\vskip 3pt
\end{document}